\documentclass[12pt,reqno]{amsart}

\usepackage{amssymb}

\newcommand{\places}{\mathsf P}
\newcommand{\dirichlet}{\mathsf d}
\newcommand{\lcm}{\operatorname{lcm}}
\newcommand{\ord}{\operatorname{ord}}
\newcommand{\fix}{\operatorname{\mathsf F}}

\newcommand{\orbit}{\operatorname{\mathsf O}}
\newcommand{\morder}{\ell}

\newcommand{\ch}{\mathrm{char}}
\newcommand{\ass}{\mathrm{Ass}}

\newcommand{\divides}{|}
\newcommand{\smallnotdivides}{\nmid}
\newcommand{\notdivides}{\nmid}

\def\bigo{\operatorname{O}}    
\def\epsilon{\varepsilon}

\renewcommand{\le}{\leqslant}
\renewcommand{\ge}{\geqslant}

\renewcommand{\mid}{\colon}
\newtheorem{theorem}{Theorem}[section]
\newtheorem{proposition}[theorem]{Proposition}
\newtheorem{lemma}[theorem]{Lemma}

\theoremstyle{definition}
\newtheorem{definition}[theorem]{Definition}
\newtheorem{example}[theorem]{Example}

\theoremstyle{remark}
\newtheorem{remark}[theorem]{Remark}

\begin{document}

\bibliographystyle{plain}

\title{Dirichlet series for finite
combinatorial rank dynamics}

\author{G. Everest}
\author{R. Miles}
\author{S. Stevens}
\author{T. Ward}

\dedicatory{Draft \today}
\address{School of Mathematics, University of East
  Anglia, Norwich NR4 7TJ, United Kingdom}
\email{t.ward@uea.ac.uk}

\thanks{This research was supported by E.P.S.R.C. grant EP/C015754/1.}

\renewcommand{\subjclassname}{\textup{2000} Mathematics Subject Classification}
\begin{abstract}
We introduce a class of group
endomorphisms -- those of finite
combinatorial rank -- exhibiting slow orbit growth.
An associated Dirichlet series
is used to obtain an exact orbit counting formula,
and in the connected
case this series is shown to have a closed rational form.
Analytic properties of the Dirichlet
series are related to orbit-growth
asymptotics: depending on the
location of the abscissa of convergence
and the degree of the pole there, various
orbit-growth asymptotics are found,
all of which are polynomially bounded.
\end{abstract}

\maketitle

\tableofcontents

\section{Introduction}

A closed orbit~$\tau$ of length~$\vert\tau\vert=n$ for
a map~$\alpha\colon X\to X$ is a set of the
form
\[
\{x,\alpha(x),\alpha^2(x),\dots,\alpha^n(x)=x\}
\]
with cardinality~$n$ (in our setting~$X$ will always be
a compact metric space and~$\alpha$ a continuous map).
Dynamical analogues of the prime number
theorem concern the asymptotic behavior of quantities like
\begin{equation*}\label{equation:orbitvcounter}
\pi_{\alpha}(N)=\left\vert\{\tau\mbox{ a closed orbit
of~$\alpha$}\mid\vert\tau\vert\le
N\}\right\vert.
\end{equation*}
When~$X$ has a metric structure with respect to which~$\alpha$ is
hyperbolic, results of Parry and
Pollicott~\cite{MR710244},~\cite{MR727704},~\cite{MR1085356} and
others apply to give a precise understanding of the growth
properties of~$\pi_{\alpha}$. Without hyperbolicity less is known:
Waddington~\cite{MR1139101} considered the case of a quasihyperbolic
toral automorphism, and the authors~\cite{emsw} found asymptotics
for connected~$S$-integer systems with~$S$ finite
(these are isometric extensions of hyperbolic
automorphisms of arithmetic origin).

Writing
\begin{equation*}
\fix_{\alpha}(n)=\vert\{x\in X\mid {\alpha}^nx=x\}\vert
\end{equation*}
for the number of points fixed by~${\alpha}^n$, the dynamical zeta
function of~$\alpha$ is defined by
\begin{equation*}\label{equation:dynamicalzetafunction}
\zeta_{{\alpha}}(z)=\exp\sum_{n=1}^{\infty}\frac{z^n}{n}\fix_{\alpha}(n)
\end{equation*}
which has a formal expansion as an Euler product,
\begin{equation*}\label{equation:eulerexpansion}
\zeta_{\alpha}(z)=\prod_{\tau}\left(1-z^{\vert\tau\vert}\right)^{-1},
\end{equation*}
where the product is taken over all closed orbits of~${\alpha}$. In the
hyperbolic and quasihyperbolic settings, the dynamical zeta function
is either rational or has a meromorphic extension beyond its radius
of convergence, allowing analytic methods to be used to obtain
asymptotic estimates for~$\pi_{\alpha}$. For~$S$-integer systems
with~$S$ finite the dynamical zeta function typically has a natural
boundary, so more direct methods have to be used in~\cite{emsw}. In
all these cases the asymptotic takes the form
\[
\pi_{\alpha}(N)\sim \frac{e^{h(\alpha)N}}{N}\Psi(N),
\]
where~$\Psi$ is an explicit almost-periodic function bounded away
from zero and infinity, constant in the hyperbolic cases,
and~$h(\alpha)$ is the topological entropy of~${\alpha}$.
The~$S$-integer systems with~$S$ finite may be viewed as
arithmetical perturbations of (quasi)hyperbolic
systems, and the orbit
growth is still essentially exponential.

Our purpose here is to describe orbit growth for a
class of dynamical systems much further from hyperbolicity:
those of \emph{finite combinatorial rank}.
The dynamical system~$(X,\alpha)$ has
finite combinatorial rank if the sequence of periodic
point counts~$(\fix_\alpha(n))$ has finite rank in a
canonical partially ordered set of integer sequences (see
Definition~\ref{definefcr}).
These systems are found to possess an orbit structure which
allows Dirichlet series to be used directly.
The simplest non-trivial
example of a group automorphism with finite combinatorial rank
is the automorphism~$\alpha$ dual to the map~$x\mapsto2x$
on the localization~$\mathbb Z_{(3)}$ of the integers at
the prime~$3$
(the
periodic point data for this map is the same as
that of the map~$z\mapsto z^2$ on~$\{z\in\mathbb C\mid z^{3^k}=1\mbox{
for some }k\geqslant0\}$).
By~\cite{MR1461206}, this
system has~$\fix_{\alpha}(n)=\vert 2^n-1\vert_3^{-1}$,
and Stangoe~\cite{stangoethesis}
used this to show the asymptotic
\[
\pi_{\alpha}(N)=\frac{1}{\log 3}\log(N)+\bigo(1),
\]
for all~$N\ge1$.

The main results concern an algebraic
system~$(X,\alpha)$ of finite combinatorial
rank, and the number~$\orbit_{\alpha}(n)$ of
closed orbits of length~$n$ under~$\alpha$.
We associate the Dirichlet series
\[
\dirichlet_{\alpha}(z)=
\sum_{n=1}^{\infty}
\frac{\orbit_{\alpha}(n)}{n^z}
\]
to~$(X,\alpha)$ and use this to study
the orbit-growth function
\[
\pi_{\alpha}(N)=\sum_{n\le N}\orbit_{\alpha}(n).
\]
In Theorem~\ref{theorem:exactorbitcountingformula}
we find an exact expression for
the Dirichlet series and use this to
show that, in the connected case,~$\dirichlet_{\alpha}(z)$
is a rational function of the
variables~$\{c^{-z}\mid c\in\mathcal C\}$
where~$\mathcal C$ is a finite set of positive
integers (see Example~\ref{exampleone} for
a simple instance of this rationality result).

The asymptotic behaviour of~$\pi_{\alpha}$
is governed by the abscissa of convergence~$\sigma$
of~$\dirichlet_{\alpha}$ and the order~$K$
of the pole at~$\sigma$.
For the case~$\sigma=0$,
Theorem~\ref{wheniwakeupearlyinthemorning}(1)
shows that
\[
\pi_{\alpha}(N)=C\left(\log
N\right)^K+\bigo\left((\log N)^{{K-1}}\right).
\]
Example~\ref{threecapfive} illustrates this
result.
For~$\sigma>0$ the situation is
more involved. In combinatorial rank
one (in which case the pole at~$\sigma$
is necessarily simple), Theorem~\ref{rank_one_theorem}
shows that
\[
\pi_{\alpha}(N)=\delta(N)N^{\sigma}+\bigo(1)
\]
where~$\delta$ is an explicit
oscillatory function bounded away from zero and infinity.
An instance of this phenomena may be
found in Example~\ref{oscillatorytermexample}.

For the general case of a simple
pole, we show
in Theorem~\ref{wheniwakeupearlyinthemorning}(2)
a Chebychev result of the form
\[
A N^{\sigma}\leqslant\pi_\alpha(N)\leqslant B N^{\sigma}
\]
for constants~$A,B>0$. Given the
oscillatory function that arises
in combinatorial rank one, no stronger
asymptotic can be expected.

Finally, and surprisingly, in the
case~$\sigma>0$ and~$K\ge2$,
Theorem~\ref{wheniwakeupearlyinthemorning}(3)
gives the exact asymptotic
\[
\pi_{\alpha}(N)\sim CN^{\sigma}\left(\log N\right)^{K-1}.
\]
This comes about because the higher-order
pole introduces a factor which
randomizes the oscillatory
behaviour seen in the case of a simple
pole. This is most easily
seen in Example~\ref{theend?}.

Typically, particularly in
analytic number theory, the
route from the singular behaviour
of a Dirichlet series to a counting
function goes via a Tauberian theorem.
However, the presence of singularities
of~$\dirichlet_\alpha$
arbitrarily close together along the
line~$\Re(z)=\sigma$
hampers this approach. Indeed,
the oscillatory term in
Theorem~\ref{rank_one_theorem} is
incompatible with the conclusion
one might hope to reach via Tauberian methods.
In simpler situations like
Example~\ref{exampleone}
(where the singularites on the line~$\Re(z)=\sigma$
are equally spaced), Agmon's
Tauberian theorem~\cite{MR0054079} or
Perron's Theorem~\cite[Th.~13]{MR0185094}
can be applied.
However, it is far easier to obtain
asymptotics directly from the series expression
for~$\dirichlet_\alpha$.

The letter~$C$ is used to denote
various constants
independent of variables~$n,N,x$ and so on;~$\mathbb N_0$
denotes~$\mathbb N\cup\{0\}$.

\section{Finite combinatorial rank}
\label{section:definesthesystems}

Let~$X$ be a compact metrizable abelian group of finite topological
dimension and let~$\alpha$ be a continuous ergodic
epimorphism of~$X$ with finite topological entropy.
Let~$\mathcal{D}$ denote the class of such dynamical
systems~$(X,\alpha)$. For~$(X,\alpha)\in\mathcal{D}$,
Lemma~\ref{counting_formula_lemma} below will show
that~$\fix_{\alpha}(n)$ is finite for all~$n\geqslant 1$.

Consider the ring~$\mathbb{Z}^{\mathbb{N}}$ with the product
ordering~$\preccurlyeq$
(so~$(a_n)\preccurlyeq(b_n)$ if~$a_n\le b_n$ for
all~$n\in\mathbb Z$). For a subset~$A\subset\mathbb Z^{\mathbb
N}$,~$a\in A$ is said to have \emph{finite rank in~$A$} if there is
some~$C>0$ such that all strictly decreasing chains of elements
of~$A$ starting at~$a$ have length at most~$C$.

Pontryagin duality gives a one-to-one correspondence between
elements of~$\mathcal{D}$ and certain countable modules over the
domain~$R=\mathbb{Z}[t]$ as follows.
The additive group~$M=\widehat{X}$ has the
structure of an~$R$-module by identifying~$x\mapsto tx$ with the
map~$\widehat{\alpha}$ and extending in an obvious way to
polynomials. The module~$M$ satisfies the following conditions:
\begin{enumerate}
\item $M$ is countable (since~$X$ is metrizable);
\item the map~$x\mapsto tx$ is a monomorphism of~$M$ (since~$\alpha$
is onto);
\item the set of associated primes~$\ass(M)$ is finite
and consists entirely of non-zero principal ideals, none of which
are generated by cyclotomic polynomials (since~$X$ is
finite-dimensional,~$\alpha$ is ergodic and~$h(\alpha)<\infty$);
\item for each~$\mathfrak{p}\in\ass(M)$,
\[
m(\mathfrak{p})=\dim_{\mathbb{K}(\mathfrak{p})}M_{\mathfrak{p}}<\infty,
\]
where~$\mathbb{K}(\mathfrak{p})$ denotes the field of fractions
of~$R/\mathfrak{p}$ (since $X$ is finite-dimensional and~$h(\alpha)<\infty$).
\end{enumerate}
Conversely, if~$M$ is an~$R$-module satisfying these four properties
and~$\alpha_M$ denotes the epimorphism of~$X_M=\widehat{M}$ dual
to~$x\mapsto tx$ on~$M$, then
\[
(X_M,\alpha_M)\in\mathcal{D}.
\]
Denote the class of all such~$R$-modules~$M$ by~$\widehat{\mathcal{D}}$.

The next lemma gives a formula for the number of periodic points of
an element of~$\mathcal{D}$ in terms of valuations of sequences of
the form
\[
\theta(\mathfrak{p})=(\overline{t}^n-1)_{n\geqslant 1},
\]
where~$\overline{t}$ denotes the image of~$t$ in~$R/\mathfrak{p}$
for some prime ideal~$\mathfrak{p}$ associated to~$M$. In order to describe
this, suppose that~$\mathfrak{p}\subset R$ is a principal prime
ideal so~$\mathbb{K}(\mathfrak{p})$ is a global field with finite
residue class fields.
Let~$\places(\mathfrak{p})$,~$\places_0(\mathfrak{p})$,~$\places_\infty(\mathfrak{p})$
denote the places, finite places and infinite places
of~$\mathbb{K}(\mathfrak{p})$ respectively. In what follows, if~$v$
denotes a normalized additive valuation corresponding to a place
in~$\places_0(\mathfrak{p})$,~$|\cdot|_v$ is defined
by~$|\cdot|_v=|\mathfrak{K}_v|^{-v(\cdot)}$, where~$\mathfrak{K}_v$
is the residue class field. For ease of notation, if~$P$ is any
subset of the set of places of a global field~$K$ then write
\[
|x|_P=\prod_{v\in P}|x|_v,
\]
for~$x\in K$, with the convention that~$|x|_{\varnothing}=1$.
Finally, note that each~$v\in\places_0(\mathfrak{p})$ also induces a
map~$|\cdot|_v\colon\mathbb{K}(\mathfrak{p})^{\mathbb{N}}
\rightarrow\mathbb{Q}^{\mathbb{N}}$,
and~$|\theta(\mathfrak{p})|_v$ is a unit in~$\mathbb{Q}^{\mathbb{N}}$
provided~$\mathfrak{p}$ is not generated by a cyclotomic polynomial.

Write~$\fix_{\alpha}$ for the sequence~$(\fix_{\alpha}(n))_{n\ge1}$,
and for sequences~$a=(a_n)$ and~$b=(b_n)$ write~$ab,a^k$ for the
sequences~$(a_nb_n),(a_n^k)$ respectively.

\begin{lemma}\label{counting_formula_lemma}
Let~$M\in\widehat{\mathcal{D}}$. Then there
exist~$\mathfrak{p}_1,\dots,\mathfrak{p}_r\in\ass(M)$ and sets of
 places~$P_1,\dots,P_r$
with~$P_i\subset\places_0(\mathfrak{p}_i)$,~$1\leqslant i\leqslant r$,
such that
\begin{equation}\label{periodic_point_formula}
\fix_{\alpha_M}=\prod_{i=1}^{r}|\theta(\mathfrak{p}_i)|_{P_i}^{-1}.
\end{equation}
\end{lemma}

\begin{proof}
This is shown in~\cite[Th.~1.1]{miles_periodic_points}.
\end{proof}

For~$M\in\widehat{\mathcal{D}}$, set
\[
[M] = \{\fix_{\alpha_L}\mid L\mbox{ is a submodule of }
M\}\subset\mathbb{Z}^{\mathbb{N}}.
\]
Lemma~\ref{counting_formula_lemma} and the next result together show
how~$[M]$ relates to the simpler constituent sets of sequences of
the form~$[\mathbb{K}(\mathfrak{p})]$. The members of a set
like~$[\mathbb{K}(\mathfrak{p})]$ are sequences of periodic points
for~$S$-integer systems.
Write~$2^P$ for the set of subsets of~$P$.

\begin{lemma}\label{sets_of_places_to_sequences_lemma}
Let~$\mathfrak{p}\subset R$ be a principal prime ideal which is not
generated by a cyclotomic polynomial, and set
\[
P=\{v\in\places_0(\mathfrak{p})\mid\left\vert
R/\mathfrak{p}\right\vert_v\mbox{ is bounded}\}.
\]
Then~$[\mathbb{K}(\mathfrak{p})]$ is the image of the
map~$\psi\colon 2^P\rightarrow\mathbb{Z}^{\mathbb{N}}$
where
\[
\psi(Q)=\vert\theta(\mathfrak{p})\vert^{-1}_Q.
\]
\end{lemma}

\begin{proof}
The proof of~\cite[Th.~1.1]{miles_periodic_points}
shows
that~$[\mathbb{K}(\mathfrak{p})]\subset\psi(2^P)$, so it remains to
show the reverse inclusion. Fix a subset~$Q\subset P$.
If~$Q=\varnothing$, set~$L=\mathbb{K}(\mathfrak{p})$, otherwise
set~$L=\bigcap_{v\in Q}R_v$, where~$R_v$ is the discrete valuation
ring of~$\mathbb{K}(\mathfrak{p})$ corresponding to~$v$. Then~$L$ is
a submodule of~$\mathbb{K}(\mathfrak{p})$
with~$\fix_{\alpha_L}=\psi(Q)$ by construction.
\end{proof}

It follows that
\begin{equation}\label{canonical_inclusion_of_ordered_sequences}
[M]\subset\prod_{\mathfrak{p}\in\ass(M)}[\mathbb{K}(\mathfrak{p})]^{m(\mathfrak{p})}
\end{equation}
and each~$[\mathbb{K}(\mathfrak{p})]$ has a description as in
Lemma~\ref{sets_of_places_to_sequences_lemma}.

\begin{proposition}
For any~$R$-module~$M\in\widehat{\mathcal{D}}$ the
following holds.
\begin{enumerate}
\item For each~$\mathfrak{p}\in\ass(M)$, the
set~$[\mathbb{K}(\mathfrak{p})]$ contains a greatest element
\[
s(\mathfrak{p})=\fix_{\alpha_{R/\mathfrak{p}}}.
\]
Furthermore, for any~$a\in
[\mathbb{K}(\mathfrak{p})]$,~$a^{-1}s(\mathfrak{p})\in
[\mathbb{K}(\mathfrak{p})]$.
\item The set~$[M]$ contains a greatest element
\[
\fix_{\alpha_L}=\prod_{\mathfrak{p}\in\ass(M)}s(\mathfrak{p})^{m(\mathfrak{p})},
\]
given by a Noetherian submodule~$L\subset M$.
\end{enumerate}
\end{proposition}

\begin{proof}
(1) This follows from the description given by
Lemma~\ref{sets_of_places_to_sequences_lemma}: The
sequence~$\psi(P)=\fix_{\alpha_{R/\mathfrak{p}}}$ is the greatest
element of~$[\mathbb{K}(\mathfrak{p})]$ and for any
subset~$Q\subset P$, we
have~$\psi(Q)\psi(P\setminus Q)=\psi(P)$.

(2) Let
\[
s=\prod_{\mathfrak{p}\in\ass(M)}s(\mathfrak{p})^{m(\mathfrak{p})},
\]
then~(1) and Lemma~\ref{counting_formula_lemma} show
that~$\fix_{\alpha_M}\preccurlyeq s$. The result follows by
observing that there is a Noetherian submodule~$L\subset M$ with a
prime filtration in which each of the
primes~$\mathfrak{p}\in\ass(M)$ appears with
multiplicity~$m(\mathfrak{p})$ and no other primes appear. It
follows that~$\fix_{\alpha_L}=s$.
\end{proof}

\begin{example}
If~$\alpha_M$ is a quasihyperbolic toral epimorphism
then~$\fix_{\alpha_M}$ is the greatest element of~$[M]$. More
generally, for any Noetherian
module~$M\in\widehat{\mathcal{D}}$,~$\fix_{\alpha_M}$ is the
greatest element of~$[M]$.
\end{example}

The next result gives a convenient characterization of the finite-rank
sequences in a poset of the form~$[\mathbb{K}(\mathfrak{p})]$;
this will be useful when considering dynamical systems of finite
combinatorial rank in the next section.

\begin{theorem}\label{finite_rank_to_finite_places_theorem}
Let~$\mathfrak{p}\subset R$ be a principal prime ideal which is not
generated by a cyclotomic polynomial and
let~$a\in[\mathbb{K}(\mathfrak{p})]$. Then~$a$ has finite rank if
and only if there is a finite set of
places~$Q\subset\places_0(\mathfrak{p})$ such
that~$|\theta(\mathfrak{p})|_Q^{-1}=a$.
\end{theorem}

The proof of Theorem~\ref{finite_rank_to_finite_places_theorem}
requires the following lemma, taken
from~\cite[Lem.~4.9]{miles_zeta}. Recall that~$\mathfrak{K}_v$
denotes the residue class field of a place~$v$, write~$\morder_v$
for the multiplicative order of the image of~$t$
in~$\mathfrak{K}_v^\times$ and $r_v$ for the
(absolute) residue degree.

\begin{lemma}\label{fundamental_evaluation_lemma}
Let~$\mathfrak{p}\subset R$ be a principal prime ideal which is not
generated by a cyclotomic polynomial.
Let~$v\in\places_0(\mathfrak{p})$ have the property
that~$|\overline{t}|_v=1$ and let~$p=\ch(\mathfrak{K}_v)$. Then
there exist constants~$D\geqslant 1$, $E\geqslant 0$ such that
\[
|\theta(\mathfrak{p})_n|^{-1}_v = \left\{
\begin{array}{ll}
1 & \textrm{ if } \morder_v\notdivides n, \\
Dp^{r_v\ord_p(n)} & \textrm{ if } \morder_v\divides n,
\ch(\mathbb{K}(\mathfrak{p}))=0
\mbox{ and } \ord_p(n)\geqslant E,\\
|\theta(\mathfrak{p})_{\morder_v}|_v^{-\vert n \vert_p^{-1}} &
\textrm{ if } \morder_v\divides n \mbox{ and }
\ch(\mathbb{K}(\mathfrak{p}))>0.
\end{array}
\right.
\]
\end{lemma}

\begin{proof}[Proof of Theorem \ref{finite_rank_to_finite_places_theorem}]
Let~$\mathbb{K}=\mathbb{K}(\mathfrak{p})$,~$\theta=\theta(\mathfrak{p})$
and let~$P$,~$\psi$ be as in Lemma
\ref{sets_of_places_to_sequences_lemma}. For any~$Q\subset P$, set
\[
\widetilde{Q}=\{v\in Q\mid |\theta|^{-1}_v\neq 1\}.
\]
Notice that if~$|\overline{t}|_v<1$ then~$|\theta|_v=1$, so
$|\overline{t}|_v=1$ for all~$v\in\widetilde{P}$.

Suppose that~$a$ has finite combinatorial rank and that~$Q\subset P$
satisfies~$\psi(Q)=a$. Assume~$\widetilde{Q}$ is infinite. By choosing a
strictly decreasing chain
\[
\widetilde{Q}=Q_1\supset Q_2\supset\cdots,
\]
it follows that
\[
a=\psi(Q_1)\succ\psi(Q_2)\succ\cdots,
\]
contradicting the
assumption that~$a$ has finite
combinatorial rank.
Hence~$\widetilde{Q}$ is finite and satisfies~$\psi(\widetilde{Q})=a$.

Conversely, suppose there is a finite set~$Q\subset P$
with~$\psi(Q)=a$. Without loss of generality, assume
that~$Q=\widetilde{Q}$. If there are only finitely many
distinct~$Q'\subset \widetilde{P}$ with~$\psi(Q')\preccurlyeq a$,
the required result follows. Hence, for a contradiction assume there
are infinitely many such~$Q'$.

First suppose that~$\ch(\mathbb{K})=0$. Then the set of rational
primes
\[
S=\{\ch(\mathfrak{K}_v)\mid v\in Q\}
\]
is finite. Furthermore, using
Lemma~\ref{fundamental_evaluation_lemma}, there exist positive
constants~$c,D,E$ such that for each~$n\geqslant 1$
with~$\min_{p\in S}\{\ord_p(n)\}\geqslant E$,
\begin{equation}\label{finite_rank_equivalence_zero_char_bounding_equation}
a_n=|\theta_n|^{-1}_Q\leqslant Dc^{j(n)},
\end{equation}
where~$j(n)=\max\{\ord_p(n)\mid p\in S\}$.

There are only finitely many~$v\in\places_0(\mathbb{K})$
with~$\ch(\mathfrak{K}_v)$ equal to a given prime. So, since there
are infinitely many distinct~$Q'\subset \widetilde{P}$
with~$\psi(Q')\preccurlyeq a$, it is possible to choose~$Q'$ to
contain a place~$v$ with~$\ch(\mathfrak{K}_v)\not\in S$.
For~$r\geqslant 0$, set
\[
n(r)=\morder_v q^r\prod_{p\in S}p^E,
\]
where~$q=\ch(\mathfrak{K}_v)$. Since~$j(n(r))$ is
bounded,~\eqref{finite_rank_equivalence_zero_char_bounding_equation}
shows that~$a_{n(r)}$ is also bounded. However, again applying
Lemma~\ref{fundamental_evaluation_lemma}, for a sufficiently large
choice of~$r$,
\[
|\theta_{n(r)}|^{-1}_{Q'}\geqslant|\theta_{n(r)}|_v^{-1}>a_{n(r)},
\]
contradicting~$\psi(Q')\preccurlyeq a$.

Now assume that~$\ch(\mathbb{K})=p>0$.
Then~$R/\mathfrak{p}=\mathbb{F}_p[\overline{t}]$ and~$\overline{t}$
is transcendental over~$\mathbb{F}_p$. Since~$Q$ is finite,
Lemma~\ref{fundamental_evaluation_lemma} shows that the
set~$\{a_n\mid\ord_p(n)=0,n\geqslant 1\}$ is bounded. There are, by
assumption, infinitely many distinct~$Q'\subset \widetilde{P}$
with~$\psi(Q')\preccurlyeq a$ so it is possible to choose such
a~$Q'$ to contain a place~$v$ for which~$|\mathfrak{K}_v|$ is larger
than this bound. However,~$|\theta_{\morder_v}|_v \leqslant
|\mathfrak{K}_v|^{-1}$ and since~$\ord_p(\morder_v)=0$,
\[
|\theta_{\morder_v}|^{-1}_{Q'} \geqslant |\theta_{\morder_v}|_v^{-1}\geqslant
|\mathfrak{K}_v|>a_{\morder_v},
\]
contradicting~$\psi(Q')\preccurlyeq a$ again.
\end{proof}

\begin{figure}[!h]
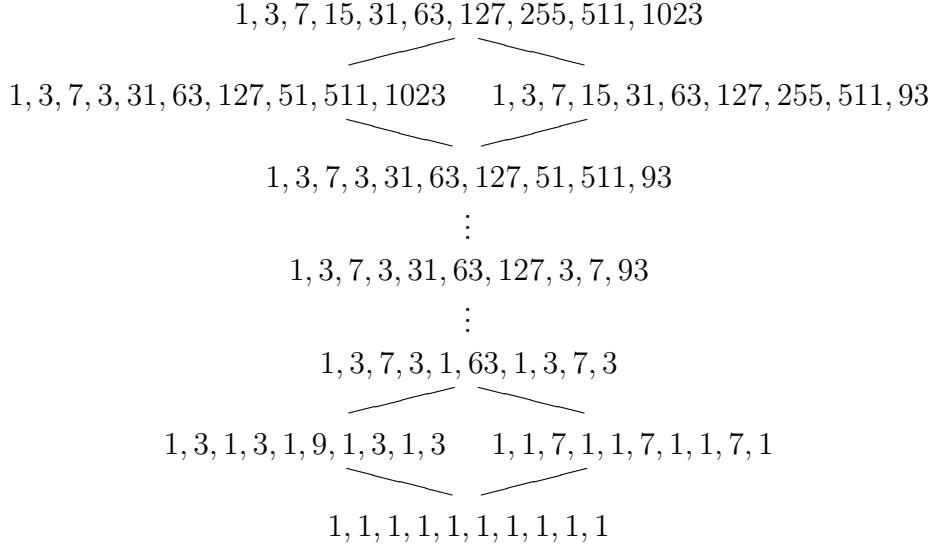

\centering
\renewcommand{\arraystretch}{1.1}
\begin{tabular}{rl}
\multicolumn{2}{c}{$1,3,7,15,31,63,127,255,511,1023$}\\
\line(4,1){40} & \hspace{39pt}\line(-4,1){40}\\
$1,3,7,3,31,63,127,51,511,1023~$ & $~1,3,7,15,31,63,127,255,511,93$\\
\line(-4,1){40} \hspace{27pt} & \line(4,1){40}\\
\multicolumn{2}{c}{$1,3,7,3,31,63,127,51,511,93$}\\
\multicolumn{2}{c}{$\vdots$}\\
\multicolumn{2}{c}{$1,3,7,3,31,63,127,3,7,93$}\\
\multicolumn{2}{c}{$\vdots$}\\
\multicolumn{2}{c}{$1,3,7,3,1,63,1,3,7,3$}\\
\line(4,1){40} & \hspace{39pt}\line(-4,1){40}\\
$1,3,1,3,1,9,1,3,1,3~$ & $~1,1,7,1,1,7,1,1,7,1$\\
\line(-4,1){40} \hspace{27pt} & \line(4,1){40}\\
\multicolumn{2}{c}{$1,1,1,1,1,1,1,1,1,1$}\\
\end{tabular}
\caption{\label{hassediagramfordoublingmap}Part of the Hasse diagram
for~$[\mathbb{K}(t-2)]$.}
\end{figure}

\begin{example}
The poset~$[\mathbb{K}(t-2)]$, which contains the sequences of
periodic point counts arising from algebraic factors of the map~$x\mapsto
2x$ on the solenoid~$\widehat{\mathbb Q}$, is illustrated in
Figure~\ref{hassediagramfordoublingmap} (the figure shows a small
part of the full poset: the first level from the bottom
has infinitely many sequences, parameterised by the
rational primes). The greatest element
corresponds to the circle doubling map~$x\mapsto 2x$ on~$\mathbb
T=\mathbb R/\mathbb Z$, and the central sequence
represents~$\psi(Q)$, where~$Q$ is the set of Mersenne primes.
According to
Theorem~\ref{finite_rank_to_finite_places_theorem},~$\psi(Q)$
has finite rank if and only if there are finitely many Mersenne
primes. The least element corresponds to the map~$x\mapsto 2x$
on~$\widehat{\mathbb Q}$, corresponding to~$\psi(Q)$ where~$Q$ is
the set of all rational primes.
\end{example}

\begin{remark}
The~$S$-integer systems studied by Chothi, Everest and
Ward~\cite{MR1461206} give an alternative way to describe the
periodic point sequences we study. In the connected case, they may
be described as follows. Fix an algebraic number field~$\mathbb K$
with set of places~$\places(\mathbb K)$ and set of infinite
places~$\places_{\infty}(\mathbb K)$, an element of infinite
multiplicative order~$\xi\in\mathbb K^*$, and a set~$S\subset
\places(\mathbb K)\setminus\places_{\infty}(\mathbb K)$ with the
property that~$\vert\xi\vert_w\le1$ for all~$w\notin S\cup
\places_{\infty}(\mathbb K)$. The associated ring of~$S$-integers is
\[
R_S=\{x\in\mathbb K\mid\vert x\vert_w\le1\mbox{ for all } w\notin
S\cup\places_{\infty}(\mathbb K)\}.
\]
Let~$X$ be the character group of~$R_S$, and define an
endomorphism~$\alpha$ to be the dual of the map~$x\mapsto\xi x$ on~$R_S$.
By~\cite[Lemma~5.2]{MR1461206} the number of points in~$X$ fixed
by~$\alpha^n$ is
\begin{equation*}\label{equation:periodicpointformulaingeneral}
\fix_{\alpha}(n)=\prod_{w\in S\cup\places_{\infty}(\mathbb K)}\vert
\xi^n-1\vert_w.
\end{equation*}
In the language of~\cite{MR1461206}, the systems we are interested
in are the co-finite ones (those for which~$\places(\mathbb
K)\setminus S$ is a finite set) and finite products of them. Using
the product formula for global fields, the periodic point formula
when~$S$ is co-finite reduces to one of the factors
in~\eqref{periodic_point_formula}.
\end{remark}

In view of the canonical inclusion
\eqref{canonical_inclusion_of_ordered_sequences} and the results
of this section, we make the following definition.

\begin{definition}\label{definefcr}
Let~$M\in\widehat{\mathcal{D}}$.
Then~$\alpha_M$ has \emph{finite combinatorial rank}
if~$\fix_{\alpha_M}$ has finite rank in
\[
\prod_{\mathfrak{p}\in\ass(M)}[\mathbb{K}(\mathfrak{p})]^{m(\mathfrak{p})}.
\]
\end{definition}
For the rest of the paper, we assume
that~$(X,\alpha)\in\mathcal{D}$.

\section{Dynamical Dirichlet series}

Writing~$\orbit_{\alpha}(n)$ for the number of orbits of length~$n$
under~${\alpha}$, we have the following arithmetic relation between periodic points and orbits,
\[
\fix_{\alpha}(n)=\sum_{d\divides n}d\orbit_{\alpha}(d).
\]
By M\"obius inversion,
\begin{equation}\label{equation:moebius_inversion_formula}
\orbit_{\alpha}(n)=\frac{1}{n}\sum_{d\divides
n}\mu(n/d)\fix_{\alpha}(d).
\end{equation}

\begin{definition}
The \emph{dynamical Dirichlet series} associated to the
map~${\alpha}$ is the formal series
\[
\dirichlet_{\alpha}(z)=\sum_{n=1}^{\infty}\frac{\orbit_{\alpha}(n)}{n^z}.
\]
\end{definition}

\begin{example}
The examples we wish to study are
group epimorphisms of finite combinatorial rank,
but the quadratic map~$\alpha:x\mapsto1-cx^2$ on the
interval~$[-1,1]$ at the Feigenbaum
value~$c=1.401155\cdots$ (see
Feigenbaum's lecture
notes~\cite{MR1209847}) gives a particularly
simple example of a dynamical
Dirichlet series. This map has
\[
\orbit_{\alpha}(n)=\begin{cases}1&\mbox{if }n=2^k
\mbox{ for some }k\ge0;\\
0&\mbox{if not},
\end{cases}
\]
so~$\dirichlet(z) =\frac{1}{1-2^{-z}}$. In this
example it is clear that~$\pi_{\alpha}(N)=\frac{\log N}{\log 2}+\bigo(1)$
in accordance with Theorem~\ref{wheniwakeupearlyinthemorning},
but it is important to emphasise that
in general the asymptotic growth statements
cannot be deduced from
the
analytic behaviour of the
Dirichlet series alone.
\end{example}

Using convolution of Dirichlet series
(see~\cite[Sec.~3.7]{MR1442260}), an alternative way of expressing
the relation~\eqref{equation:moebius_inversion_formula} is
\begin{equation}\label{equation:zetatrickone}
\dirichlet_\alpha(z)=\frac{1}{\zeta(z+1)}\sum_{n=1}^{\infty}\frac{\fix_\alpha(n)/n}{n^z},
\end{equation}
where as usual
\[
\zeta(z+1)=\sum_{n=1}^{\infty}\frac{1/n}{n^z}.
\]
For systems of finite combinatorial rank, this observation is
extremely useful: it is possible to extract~$\zeta(z+1)$ as a factor
of the series on the right-hand side of
\eqref{equation:zetatrickone}, resulting in an exact expression for
orbit counting. Note that only combinatorial properties of the zeta
function are of importance here.

Recall that the set of rational primes $p$ with
\begin{equation}\label{primes_in_the_periodic_point_sequence}
p\,\divides\fix_\alpha(n)\mbox{ for some }n\in\mathbb{N}
\end{equation}
is finite when $\fix_\alpha$ has finite rank.

\begin{theorem}\label{theorem:exactorbitcountingformula}
Let $(X,\alpha)$ be a system of finite combinatorial rank
and let $Q$ be the set of rational primes given
by~\eqref{primes_in_the_periodic_point_sequence}.
Then $\dirichlet_\alpha(z)$ is a linear combination
of Dirichlet series of the form
\[
\sum_{\mathbf{e}\in \mathbb{N}_0^P}
\frac{\varphi_P(\Lambda(\mathbf{e}))}{(b\varphi_P(\mathbf{e}))^z},
\]
where $b\in\mathbb{N}$, $P\subset Q$,
$\Lambda:\mathbb{N}_0^P\rightarrow \mathbb{N}_0^P$
and $\varphi_P(\mathbf{e})=\prod_{p\in P}p^{e_p}$.

Furthermore, all these quantities can be
determined explicitly.
In particular, denoting the local field
corresponding
to a place~$v$ by~$\mathbb{K}_v$ and the residue
degree of~$v$ by~$r_v$,
\[
\Lambda(\mathbf{e})_p=-e_p+\sum_{v\in S(p)}r_v\lambda_v(e_p),
\]
where~$S(p)$ is a set of places with
residue class fields
of characteristic~$p$, and
\[
\lambda_v(e_p)=
\left\{
\begin{array}{ll}
e_p & \mbox{ if } \ch(\mathbb{K}_v) = 0, \\
\eta_v p^{e_p} & \mbox{ if } \ch(\mathbb{K}_v) > 0,
\end{array}
\right.
\]
for some positive integer constant $\eta_v$.
\end{theorem}

Theorem~\ref{theorem:exactorbitcountingformula} and its proof allow us
to deduce the following result for connected systems.

\begin{theorem}\label{theorem:dirichletseriesarerational}
Let~$(X,\alpha)$ be a connected system of finite combinatorial rank.
Then there is a finite set~$\mathcal{C}\subset\mathbb{N}$
with the property that~$\dirichlet_{\alpha}(z)$ is a rational
function of the variables~$\{c^{-z}\mid c\in\mathcal{C}\}$.
\end{theorem}

In principle, Theorem~\ref{theorem:exactorbitcountingformula} provides
an exact formula
for orbit counting, since it implies that~$\pi_\alpha(N)$
is a linear combination of expressions of the form
\begin{equation*}
\mathop{\sum_{\mathbf{e}\in \mathbb{N}_0^P,}}_{\varphi_P(\mathbf{e})\leqslant N/b}
\varphi_P(\Lambda(\mathbf{e})).
\end{equation*}
However, a closed asymptotic expression involving elementary functions
is often more desirable, the prime number theorem being a case in point.
Our focus is on connected systems,
although Theorem~\ref{theorem:exactorbitcountingformula}
may also be applied in the disconnected case
(\emph{cf}.~Example~\ref{zero_dimensional_example}).

\begin{theorem}\label{rank_one_theorem}
Let~$(X,\alpha)$ be a connected system of combinatorial rank one and
let $r_v$ be the residue degree of the associated place. Then
\[
\pi_{\alpha}(N)=
\left\{
\begin{array}{ll}
C\log N + \bigo(1) & \mbox{ if } r_v=1,\\
\delta(N)N^{r_v-1}+\bigo(1) & \mbox{otherwise},
\end{array}
\right.
\]
where~$\delta(N)$ is an
explicit
oscillatory function bounded away from zero and infinity.
\end{theorem}

Theorem~\ref{rank_one_theorem} shows that an elementary asymptotic
formula for $\pi_\alpha$ cannot be expected in general. Typically, an oscillatory function
similar to that in Theorem~\ref{rank_one_theorem} appears when the abscissa
of convergence of~$\dirichlet_\alpha$ is greater than zero and is a simple pole; we
provide a Chebychev result in this case. In all other cases,
an elementary asymptotic formula is found.

\begin{theorem}\label{wheniwakeupearlyinthemorning}
Let~$(X,\alpha)$ be a connected system of finite
combinatorial rank whose Dirichlet series~$\dirichlet_\alpha$
has abscissa of convergence~$\sigma\ge 0$. Let~$K$ denote the order of this pole.
\begin{enumerate}
\item If~$\sigma=0$ then there is a constant~$C>0$ with
\[
\pi_{\alpha}(N)
=
C\left(\log N\right)^K+\bigo\left((\log N)^{K-1}\right).
\]
\item If~$\sigma>0$ and~$K=1$ then
there are constants~$A,B,N_0>0$ such that for all $N>N_0$,
\[
A N^{\sigma}\leqslant\pi_\alpha(N)\leqslant B N^{\sigma}.
\]
\item If~$\sigma>0$ and~$K\ge 2$ then
there is a constant~$C>0$ with
\[
\pi_{\alpha}(N)
\sim
C N^{\sigma}\left(\log N\right)^{K-1}.
\]
\end{enumerate}
\end{theorem}

\begin{remark}
(1) If~$(X,\alpha)$ is a connected system of finite
combinatorial rank and~$X$ has topological dimension one, the
proof of Theorem~\ref{theorem:exactorbitcountingformula} shows that
$\dirichlet_\alpha$ has abscissa of convergence $\sigma=0$. Thus, the exact
asymptotic of Theorem~\ref{wheniwakeupearlyinthemorning}(1) applies. In general, there is no
straightforward relationship between $\dirichlet_\alpha$ and the topological dimension of~$X$.

\noindent(2) Using
Hlawka's bounds~\cite{MR0174544} on equidistribution in terms of
discrepancy and Baker's theorem, the
asymptotic in Theorem~\ref{wheniwakeupearlyinthemorning}(3)
could potentially be improved to give an error term.
\end{remark}

\section{Examples}

One of the characteristic features of non-trivial~$S$-integer
dynamical systems is the extremely complex behavior of their
dynamical zeta functions
(see~\cite{emsw},~\cite{MR2180241},~\cite{MR1702897} for example),
so the rationality
of the Dirichlet series
for systems of finite combinatorial rank
is a little surprising.

\begin{example}\label{exampleone}
The basic reason for rationality of the Dirichlet series is already
visible in the simplest non-trivial example given by the
map~$\alpha$ dual to the map~$x\mapsto2x$ on~$\mathbb
Z_{(3)}=\mathbb Z[\frac{1}{p}\mid p\mbox{ a prime}\neq3]$.
We have
\begin{equation*}
\fix_{\alpha}(n)=\vert
2^n-1\vert_3^{-1}
\end{equation*}
and hence
\[
\fix_{\alpha}(n)=\begin{cases} 3^{\ord_3(n)
+1}&\mbox{if $n$ is even};\\
1&\mbox{otherwise}.
\end{cases}
\]
From~\eqref{equation:zetatrickone} we know that
\begin{equation*}
\dirichlet_{\alpha}(z)=\zeta(z+1)^{-1}\left(
\sum_{2\smallnotdivides n}\frac{\fix_{\alpha}(n)}{n^{z+1}}
+\sum_{2\divides n}\frac{\fix_{\alpha}(n)}{n^{z+1}}\right).
\end{equation*}
Treating each term separately,
\[
\sum_{2\smallnotdivides n}\frac{\fix_{\alpha}(n)}{n^{z+1}}=
\sum_{n\ge1}\frac{1}{n^{z+1}}-\sum_{n\ge1}\frac{1}{(2n)^{z+1}}
=\zeta(z+1)\left(1-\frac{1}{2^{z+1}}\right)
\]
and (on writing an even~$n$ as~$2\cdot k\cdot 3^e$ with~$3\notdivides
k$)
\begin{eqnarray*}
\sum_{2\divides n}\frac{\fix_{\alpha}(n)}{n^{z+1}}&=&
\sum_{e\ge0}\sum_{k\ge1,3\smallnotdivides k}\frac{3^{e+1}}{(2k3^e)^{z+1}}\\
&=&
\frac{3}{2^{z+1}}\sum_{e\ge0}\frac{1}{3^{ez}}\sum_{k\ge1,3\smallnotdivides k}\frac{1}{k^{z+1}}\\
&=&
\frac{3}{2^{z+1}}\cdot\frac{1}{1-3^{-z}}\cdot\zeta(z+1)\left(
1-\frac{1}{3^{z+1}}\right).
\end{eqnarray*}
Thus,
\begin{equation*}
\dirichlet_{\alpha}(z)= 1-\frac{1}{2^{z+1}}\left(
1-\frac{3}{1-3^{-z}}\left(1-\frac{1}{3^{z+1}}\right) \right)=
1+\frac{1}{2^z}\left(\frac{1}{1-3^{-z}}\right).
\end{equation*}
Note that the abscissa of convergence is $\sigma=0$ and this is a simple pole.

The orbit-growth function $\pi_\alpha$ may be obtained by extracting the coefficients from the series expression for $\dirichlet_{\alpha}$,
\[
\pi_\alpha(N)
\;=\;
1+\sum_{e\geqslant 0,\,2\cdot 3^e\leqslant N}1
\;=\;
1+\sum_{e=0}^{\lfloor\log_3(N/2)\rfloor}1\\
\;=\;
\frac{\log N}{\log 3} + \bigo(1).
\]
\end{example}

Notice that in Example~\ref{exampleone} the rational expression for~$\dirichlet_{\alpha}$ has
infinitely many singularities along the line~$\Re(z)=0$ at the
points~$2k\pi{\rm i}/\log3$ for~$k\in\mathbb Z$. A less direct method for obtaining
an asymptotic estimate for $\pi_\alpha(N)$ is provided by Agmon's Tauberian theorem~\cite{MR0054079} which applies in  situations like this, but as the
next example shows, the general case has even worse analytic
properties. It also illustrates the impact of additional places.

\begin{example}\label{exampletwo}
Let~$\alpha$ be the map dual to~$x\mapsto 2x$ on~$\mathbb{Z}_{(3)}\cap\mathbb{Z}_{(5)}$. Then
\[
\fix_{\alpha}(n)=\vert2^n-1\vert_3^{-1}\vert2^n-1\vert_5^{-1},
\]
so
\begin{equation*}
\fix_{\alpha}(n)=\begin{cases}
1&\mbox{if $n$ is odd};\\
3\cdot 3^{e_1}&\mbox{if }n=2\cdot k\cdot 3^{e_1}, 2\notdivides k,3\notdivides k;\\
15\cdot 3^{e_1}\cdot 5^{e_2}&\mbox{if }n=4\cdot k\cdot3^{e_1}\cdot 5^{e_2},3\notdivides k,
5\notdivides k.
\end{cases}
\end{equation*}
A similar calculation to that used in Example~\ref{exampleone} shows that
\[
\dirichlet_{\alpha}(z)=1-\frac{1}{2^{z+1}}+\frac{3}{2^{z+1}}
\left(
1-\frac{1}{3^{z+1}}-\frac{1}{2^{z+1}}+\frac{1}{6^{z+1}}\right)
\frac{1}{1-3^{-z}}
\]
\[+
\frac{15}{4^{z+1}}\left(
1-\frac{1}{3^{z+1}}-\frac{1}{5^{z+1}}+\frac{1}{15^{z+1}}
\right)
\frac{1}{(1-3^{-z})(1-5^{-z})}.
\]
Here, the abscissa of convergence $\sigma=0$ is a double pole. An asymptotic expression
for $\pi_\alpha$ is obtained in Section~\ref{theyvebeengoinginandoutofstyle}.
\end{example}

Notice that in Example~\ref{exampletwo} not only are there
infinitely many singularities with~$\Re(z)=0$, but there are
singularities that are arbitrarily close together.

The next example illustrates the situation when a higher topological dimension is allowed.
This may also be regarded as allowing places with a higher residue degree.

\begin{example}\label{oscillatorytermexample}
Let~$\mathbb K=\mathbb Q(\root{3}\of{5})$, with ring
of integers
\[
\mathfrak O=
\mathbb Z+\root{3}\of{5}\mathbb Z+(\root{3}\of{5})^2\mathbb Z=
\mathbb Z[\root{3}\of{5}]
\]
as~$\mathbb Z$-modules. Notice
that~$\mathfrak O\cong R/(t^3-5)$ via the map
\[
\root{3}\of{5}\mapsto \overline{t}=t+(t^3-5).
\]
There are two primes lying above the prime~$2$ of~$\mathbb Z$,
namely~$(2,1+\overline{t})$
and~$\mathfrak{m}=(2,1+\overline{t}+\overline{t}^2)$. The place $v$
corresponding to $\mathfrak{m}$ has residue degree 2,
since~$\mathfrak O/\mathfrak m\cong\mathbb F_4$.

Let $X=\widehat{\mathfrak O_{\mathfrak m}}$ and let $\alpha$ be the map dual to $x\mapsto \root{3}\of{5}\,x$.
Then $X$ has topological dimension 3 and
\[
\fix_\alpha(n)=
\vert
\overline{t}^n-1
\vert_v=
\begin{cases}
1&3\notdivides n;\\
16\cdot 2^{2\ord_2(n)}&\mbox{otherwise.}
\end{cases}
\]
Using a similar method to the previous examples, the dynamical Dirichlet series for this example is,
\[
\dirichlet_\alpha(z)=1+\frac{5}{3^z}+\frac{8}{6^z}\left(\frac{1}{1-2^{1-z}}\right).
\]
In this case, the abscissa of convergence is $\sigma=1$ and this is a simple pole.
Using a similar method as before (or by
using \eqref{equation:moebius_inversion_formula} directly), we obtain the closed formula
\[
\orbit_{\alpha}(n)=
\begin{cases}
1&n=1;\\
5&n=3;\\
4\cdot 2^e&n=3\cdot 2^e\mbox{ and }e\ge1;\\
0&\mbox{otherwise}.
\end{cases}
\]
It follows that
\begin{eqnarray}
\nonumber
\pi_\alpha(N)
& = &
6+4\sum_{e=1}^{\lfloor\log_2(N/3)\rfloor}2^e\\
\label{stayinbedimstillyawning}
& = &
\frac{8}{3}\,2^{-\{\log_2(N/3)\}}N-2,
\end{eqnarray}
for all $N\geqslant 6$, where $\{\cdot\}$ denotes the fractional part of a real number.
Notice that~\eqref{stayinbedimstillyawning} includes the oscillatory factor~$\delta(N)$
appearing in Theorem~\ref{rank_one_theorem}. In this example,
\[
\limsup_{N\rightarrow\infty}\delta(N)=\frac{8}{3}
\]
and
\[
\liminf_{N\rightarrow\infty}\delta(N)=\frac{4}{3}.
\]
\end{example}

To conclude this section, the following example shows
that the situation is quite different for disconnected systems. Here there
is no rational expression for $\dirichlet_\alpha$ although the series
organizes orbit data in a convenient way, giving an exact expression for~$\pi_\alpha$.

\begin{example}\label{zero_dimensional_example}
Let~$M$ be the discrete valuation
ring~$\mathbb{F}_3[t]_{(t-1)}$
obtained by localizing the
domain~$\mathbb{F}_3[t]$ at the prime ideal~$(t-1)$.
Since the additive group of~$M$ is
torsion,~$X=\widehat{M}$ is zero-dimensional.
Furthermore, the map~$\alpha$ dual to~$x\mapsto tx$ has
periodic points counted by a single place~$v$ of~$\mathbb{F}_3(t)$
corresponding to~$M$. That is,
\[
\fix_\alpha(n) = |t^n-1|_{(t-1)}^{-1} = 3^{|n|_3^{-1}}.
\]
Therefore,
\begin{eqnarray*}
\dirichlet_\alpha(z)
& = &
\zeta(z+1)^{-1}\sum_{n\geqslant 1}\frac{3^{|n|_3^{-1}}}{n^{z+1}}\\
& = &
\zeta(z+1)^{-1}\sum_{e\geqslant 0}\frac{3^{3^e}}{(3^e)^{z+1}}
\sum_{k\geqslant 1,3\nmid\, k}\frac{1}{k^{z+1}}\\
& = &
\left(1-\frac{1}{3^{z+1}}\right)
\sum_{e\geqslant 0}\frac{3^{3^e}/\,3^e}{(3^e)^z.}
\end{eqnarray*}
Hence, we have the following exact expression for $\pi_\alpha$,
\begin{eqnarray*}
\pi_\alpha(N)
& = &
\sum_{e\geqslant 0,\,3^e\leqslant N}\frac{3^{3^e}}{3^e} \,-\, \frac{1}{3}\sum_{e\geqslant 0,\,3^e\leqslant N/3}\frac{3^{3^e}}{3^e}\\
& = &
1+\frac{2}{3}\sum_{e\geqslant 0,\,3^e\leqslant N}\frac{3^{3^e}}{3^e}.
\end{eqnarray*}
A further calculation shows that
\[
\limsup_{N\rightarrow\infty}\frac{N\pi_\alpha(N)}{3^N}=\frac{2}{3}
\]
and
\[
 \liminf_{N\rightarrow\infty}\frac{N\pi_\alpha(N)}{3^N}=0.
\]
\end{example}

\section{Proof of Theorems~\ref{theorem:exactorbitcountingformula}
and \ref{theorem:dirichletseriesarerational}}
\label{sectionproofofrationality}

The general case requires an inclusion--exclusion
argument to deal with the congruence conditions
arising from multiple terms in~\eqref{periodic_point_formula}
and multiple places in each set~$P_i$.

Since $(X,\alpha)$ has finite combinatorial rank,
the sets~$P_i$ appearing in~\eqref{periodic_point_formula}
are all finite. Let~$T$ be the union of the sets~$P_i$,
and let~$\xi_v$ for~$v\in T$ be the image
of~$t$ in the appropriate
field~$\mathbb K(\mathfrak p)$.
Notice that each~$\xi_v$ has infinite multiplicative
order and we may assume~$|\xi_v|_v=1$.
For any finite non-empty set~$S\subset T$
write
\begin{equation}\label{Youdontknowhowluckyyouluckyyouareboy}
f_S(n)=\prod_{v\in S}|\xi_v^n-1|_v^{-1}.
\end{equation}
Thus~$f_T(n)
=\fix_{\alpha}(n)$.
Let~$\morder_v$ be the multiplicative
order of the image of~$\xi_v$ in the residue field~$\mathfrak{K}_v$ of~$v$.
For~$S\subset T$ and a set~$Q$
of rational primes, define
\[
S(Q)=\{v\in S\mid \ch(\mathfrak{K}_v)=p\mbox{ for some } p\in Q\}\subset S.
\]
If~$Q=\{p\}$, we also write~$S(p)$ for~$S(Q)$.
Consider
\[
\dirichlet_\alpha(z)=\zeta(z+1)^{-1}\sum_{n=1}^{\infty}\frac{f_T(n)}{n^{z+1}}.
\]
Let~$\morder_S=\lcm\{\morder_v\mid v\in S\}$.
Then the collection~$\{N_S,S\subset T\}$, defined by
\[
N_S=
\{n\in\mathbb{N}\mid\morder_S\divides
n\mbox{ and }\morder_{v}\notdivides n\mbox{ forall }v\in T\setminus S\},
\]
forms a partition of $\mathbb{N}$. Hence,
\begin{equation}\label{equation:eleanorrigby}
\dirichlet_\alpha(z)=\zeta(z+1)^{-1}\sum_{S\subset T}\sum_{n \in N_S}\frac{f_S(n)}{n^{z+1}},
\end{equation}
since~$f_S=f_T$ on~$N_S$ by Lemma~\ref{fundamental_evaluation_lemma}.

Let~$E=\mathbb{N}_0^{Q}$,
where
\begin{equation*}
Q=Q(S)=\{\ch(\mathfrak{K}_v)\mid v\in S\}.
\end{equation*}
For any $P\subset Q$, we extend the definition
of $\varphi_P$ given in the statement
of Theorem~\ref{theorem:exactorbitcountingformula} to
accomodate vectors $\mathbf{e}=(e_p)\in E$ by setting
\[
\varphi_P(\mathbf{e})=\prod_{p\in P}p^{e_p}.
\]
Each integer~$n\in N_S$ can be written in the form
\begin{equation}\label{integer_rewrite_equation}
n=\morder_S k \varphi_Q(\mathbf{e}),
\end{equation}
where~$k\in\mathbb{N}$ satisfies
\begin{equation}\label{permitted_integer_condition_for_primes}
p\notdivides k\mbox{ for all }p\in Q
\end{equation}
and
\begin{equation}\label{rough_permitted_integer_condition}
\morder_v\notdivides \morder_S k \varphi_Q(\mathbf{e})\mbox{ for all }v\in T\setminus S.
\end{equation}

It is more convenient to rephrase
\eqref{rough_permitted_integer_condition} according to
the following partition of~$E$. Let
\begin{equation}\label{first_value_for_large_enough_prime_powers}
e_0=\max\{\ord_p(\morder_v)\mid p\in Q, v\in T\setminus S\}
\end{equation}
and let~$\mathbf{e}_0\in E$ be the vector with each entry equal to~$e_0$.
For~$P\subset Q$, set
\[
E_P=\{\mathbf{e}\in E\mid e_p\geqslant e_0\mbox{ for all }
p\in P,e_p<e_0\mbox{ for all }p\in Q\setminus P\},
\]
so the sets~$\{E_P\mid P\subset Q\}$
partition~$E$. For each~$q$ in the finite
set~$\varphi_{Q\setminus P}(E_P)$, let
\[
E_P(q)=\{\mathbf{e}\in E_P\mid\varphi_{Q\setminus P}(\mathbf{e})=q\}.
\]
Then, for~$\mathbf{e}\in E_P(q)$,
\[
\morder_S\varphi_Q(\mathbf{e})=\morder_S \varphi_P(\mathbf{e}) q,
\]
and if~$v\in T\setminus S$,
\[
\gcd(\morder_v,\morder_S\varphi_Q(\mathbf{e}))
=
\gcd(\morder_v,\morder_S\varphi_P(\mathbf{e})q)
=
\gcd(\morder_v,\morder_S\varphi_P(\mathbf{e}_0)q).
\]
Hence, provided~$\mathbf{e}\in E_P(q)$,
condition \eqref{rough_permitted_integer_condition} becomes,
\begin{equation}\label{permitted_integer_condition_for_orders}
j\notdivides k\mbox{ for all }
j\in \{\morder_v/\gcd(\morder_v,\morder_S
\varphi_P(\mathbf{e}_0)q)\mid v\in T\setminus S\}.
\end{equation}

For~$n$ of the form \eqref{integer_rewrite_equation},
by Lemma~\ref{fundamental_evaluation_lemma},~$f_S(n)=
f_S(\morder_S \varphi_Q(\mathbf{e}))$. Furthermore,
using the partition of~$E$ given above, the inner sum in \eqref{equation:eleanorrigby} may be written
\begin{equation}\label{equation:picksuptherice}
\sum_{P\subset Q}
\sum_{q\in\varphi_{Q\setminus P}(E_P)}
\sum_{\mathbf{e}\in E_P(q)}
(\morder_S \varphi_Q(\mathbf{e}))^{-z-1}f_S(\morder_S \varphi_Q(\mathbf{e}))
\sum_{k}k^{-z-1},
\end{equation}
where~$k$ runs through all natural numbers
satisfying~\eqref{permitted_integer_condition_for_primes}
and~\eqref{permitted_integer_condition_for_orders}.
Using inclusion-exclusion, the inner sum
in~\eqref{equation:picksuptherice} is of the form
\begin{equation}\label{equation:inthechurchwhere}
\zeta(z+1)g(c^{-z},c\in \mathcal{C}),
\end{equation}
where~$\mathcal{C}\subset\mathbb{N}$ is a finite set of constants
and~$g$ is a multivariate polynomial with rational coefficients,
the monomial terms of which have total degree one.
Using~\eqref{equation:picksuptherice} and~\eqref{equation:inthechurchwhere},
by cancelling the zeta function with its reciprocal and
adjusting~$\mathcal{C}$ and~$g$ to absorb~$\morder_S^{-z-1}$, \eqref{equation:eleanorrigby} becomes
\[
\dirichlet_\alpha(z)=
\sum_{S\subset T}\sum_{P\subset Q}\sum_{q\in\varphi_{Q\setminus P}(E_P)}\dirichlet_{S,P,q}(z),
\]
where
\begin{equation}\label{equation:livesinadream}
\dirichlet_{S,P,q}(z)=
g(c^{-z},c\in \mathcal{C})
\sum_{\mathbf{e}\in E_P(q)}
\varphi_Q(\mathbf{e})^{-z-1}
f_S(\morder_S \varphi_Q(\mathbf{e})).
\end{equation}

Again using Lemma \ref{fundamental_evaluation_lemma}, for each~$p\in Q$, there exists~$e'_p$ such that
for all~$v\in S(p)$ and all~$\mathbf{e}\in E$ with~$e_p\geqslant e_p'$,
\begin{equation}\label{equation:waitsatthewindow}
f_v(\morder_S \varphi_Q(\mathbf{e}))=c_v p^{r_v\omega_v(e_p)},
\end{equation}
where~$c_v$ is a positive integer constant, $\omega_v(e_p)=e_p$ if the local field $\mathbb{K}_v$ corresponding to $v$ has zero characteristic and $\omega_v(e_p)=\eta_v' p^{e_p}$ for some constant $\eta_v' \in\mathbb{N}$, otherwise. Without loss of generality,
the steps already carried out may be performed with~$e_0$
replaced by the maximum of the value given by
\eqref{first_value_for_large_enough_prime_powers}
and~$\max\{e_p'\mid p\in Q\}$.

For any~$\mathbf{e}\in E_P(q)$,~$\varphi_Q(\mathbf{e})
=\varphi_P(\mathbf{e})q$,~so the summand in~\eqref{equation:livesinadream} is
\begin{equation}\label{equation:aweddinghasbeen}
(\varphi_P(\mathbf{e})q)^{-z-1}
f_{S(Q\setminus P)}(\morder_S q)
f_{S(P)}(\morder_S \varphi_P(\mathbf{e})).
\end{equation}
Therefore,
\[
\dirichlet_{S,P,q}(z)
=
g(c^{-z},c\in \mathcal{C})
\sum_{\mathbf{e}\in E_P(q)}
\varphi_P(\mathbf{e})^{-z-1}
f_{S(P)}(\morder_S \varphi_P(\mathbf{e})),
\]
where~$g$ and~$\mathcal{C}$ have been adjusted to accommodate the factor
\[
q^{-z-1}f_{S(Q\setminus P)}(\morder_S q),
\]
appearing in \eqref{equation:aweddinghasbeen}.
Using~\eqref{equation:waitsatthewindow}, it follows that
\begin{equation}\label{ifyoudriveacar}
\dirichlet_{S,P,q}(z)=
C\times g(c^{-z},c\in \mathcal{C})
\sum_{\mathbf{e}\in E_P(q)}
\frac{\varphi_P(\Omega(\mathbf{e})-\mathbf{e})}{\varphi_P(\mathbf{e})^z}
\end{equation}
where $C>0$ and $\Omega:\mathbb{N}_0^P\rightarrow \mathbb{N}_0^P$ is given by,
\[
\Omega(\mathbf{e})_p=\sum_{v\in S(p)}r_v\omega_v(e_p).
\]

By the definition of $E_P(q)$ and $\varphi_P$, the sum in \eqref{ifyoudriveacar} is
\begin{equation}\label{illtaxthestreet}
\mathop{\sum_{\mathbf{e}\in \mathbb{N}_0^P,}}_{e_p\geqslant e_0,\, p\in P}
\frac{\varphi_P(\Omega(\mathbf{e})-\mathbf{e})}{\varphi_P(\mathbf{e})^z}
 =
\sum_{\mathbf{e}\in \mathbb{N}_0^P}
\frac{\varphi_P(\Omega(\mathbf{e}+\mathbf{e}_0)-\mathbf{e}-\mathbf{e}_0)}
{\varphi_P(\mathbf{e}+\mathbf{e}_0)^z}
\end{equation}
where $\mathbf{e}_0\in\mathbb{N}_0^P$ is the vector with each entry
equal to $e_0$. For each $v\in S(p)$ with $\ch(\mathbb{K}_v)=p$, set $\eta_v=p^{e_0}\eta_v'$ and let  $\Lambda:\mathbb{N}_0^P\rightarrow \mathbb{N}_0^P$
be defined in terms of $\eta_v$ as in the
statement of Theorem~\ref{theorem:exactorbitcountingformula}. Then
\[
\Omega(\mathbf{e}+\mathbf{e}_0)-\mathbf{e}
=
\Lambda(\mathbf{e})+\mathbf{d},
\]
where $\mathbf{d}=(d_p)\in\mathbb{N}_0^{P}$ is given by
\[
d_p=e_0\mathop{\sum_{v\in S(p),}}_{\ch(\mathbb{K}_v)=0}r_v.
\]
Therefore, by \eqref{ifyoudriveacar} and \eqref{illtaxthestreet},
\begin{equation}\label{bangbangmaxwells}
\dirichlet_{S,P,q}(z)=
g(c^{-z},c\in \mathcal{C})
\sum_{\mathbf{e}\in \mathbb{N}_0^P}
\frac{\varphi_P(\Lambda(\mathbf{e}))}{\varphi_P(\mathbf{e})^z},
\end{equation}
where $g$ has been modified to absorb the factor
$C\varphi_P(\mathbf{d})\varphi_P(\mathbf{e}_0)^{-1-z}$. This concludes the proof of Theorem~\ref{theorem:exactorbitcountingformula}.

To prove Theorem~\ref{theorem:dirichletseriesarerational},
first note that since $X$ is assumed to be connected,
all places arise from number fields. Setting,
\[
r_{S,p}=\sum_{v\in S(p)}r_v,
\]
it follows from the definition of $\Lambda$ that
\[
\dirichlet_{S,P,q}(z)=
g(c^{-z},c\in \mathcal{C})
\sum_{\mathbf{e}\in \mathbb{N}_0^P}
\prod_{p\in P}p^{e_p(r_{S,p}-1-z)}.
\]
Therefore, the sum contains a geometric series
corresponding to each prime~$p\in P$. Hence
\begin{equation}\label{manihadadreadfulflight}
\dirichlet_{S,P,q}(z)=g(c^{-z},c\in
\mathcal{C})\prod_{p\in P}(1-p^{r_{S,p}-1-z})^{-1}.
\end{equation}
Since~$\dirichlet_\alpha(z)$ is a finite sum of expressions of this form,
this proves the theorem.

For the proof of the asymptotic results we will need
a specific description of the function~$g$ appearing
in~\eqref{bangbangmaxwells} in order to control
cancellation between terms of different signs
in various counting arguments. Notice that
for a finite set of positive integers $J$ with $1\notin J$,
\[
\sum_{I\subset J}\frac{(-1)^{\vert I\vert}}{c_I} > 0,
\]
where~$c_I=\lcm\{i\mid i\in I\}$.

\begin{lemma}\label{orshouldisay}
Suppose $g\neq 0$. Then
\[
g(c^{-z},c\in\mathcal{C})=C\times
\sum_{I\subset J}\frac{(-1)^{\vert I\vert}}{c_I}\,b_I^{-z},
\]
where~$C>0$,~$J$ is a finite set of positive integers with $1\notin J$
and~$b_I=Bc_I$ for each~$I\subset J$, for some fixed constant~$B\in\mathbb N$.
\end{lemma}

This may be seen by going through the steps in the
proof above.

\section{Proof of Theorem~\ref{rank_one_theorem}}

By
Lemma~\ref{fundamental_evaluation_lemma}, we have
constants~$\morder_v$,~$a_1,\dots,a_d$,~$D$ and a prime~$p$ with
\[
\fix_{\alpha}(n)=\begin{cases}
1&\morder_v\notdivides n;\\
a_e&n=\morder_vp^ek,\mbox{ where }p\notdivides k, 1\le e\le d;\\
Dp^{r_ve}&n=\morder_vp^ek,p\notdivides k, e>d.
\end{cases}
\]
By the method used in Section~\ref{sectionproofofrationality},
we deduce that
\[
\dirichlet_{\alpha}(z)=
1-\morder_v^{-z-1}+
\sum_{e=1}^d a_e \left(\morder_vp^e\right)^{-z-1}
+\frac{cb^{-z}}{1-p^{m-z}}\left(1-\frac{1}{p^{z+1}}\right)
\]
where~$b=\morder_vp^{d+1}$,~$c=\frac{D}{\morder_v}p^{(d+1)(r_v-1)}$
and~$m=r_v-1$.
It follows that -- up to an error uniformly bounded in~$N$ --
we can compute~$\pi_{\alpha}(N)$ by considering the dominant
term
\begin{equation}\label{letmetakeyoudowncosImgoingto}
\frac{b^{-z}}{1-p^{m-z}}\left(1-\frac{1}{p^{z+1}}\right)
=\frac{b^{-z}}{1-p^{m-z}}-
\frac{(bp)^{-z}}{p(1-p^{m-z})}.
\end{equation}
The first term in~\eqref{letmetakeyoudowncosImgoingto}
contributes
\begin{equation}\label{strawberryfields}
\sum_{e\mid bp^e\le N}p^{me}=\sum_{e=0}^{\lfloor
\log_p(N/b)\rfloor}p^{me}
\end{equation}
and the second term in~\eqref{letmetakeyoudowncosImgoingto}
contributes
\begin{eqnarray}
\frac{1}{p}\sum_{e\mid bp^{e+1}\le N}p^{me}&=&
\frac{1}{p}\sum_{e=1}^{\lfloor
\log_p(N/b)\rfloor}p^{m(e-1)}\nonumber\\
&=&p^{-m-1}\left(\sum_{e=0}^{\lfloor
\log_p(N/b)\rfloor}p^{me}\right)-p^{-m-1}.\label{repeatstrawberryfields}
\end{eqnarray}
The total contribution from~\eqref{strawberryfields}
and~\eqref{repeatstrawberryfields} is therefore
\begin{equation}\label{nothingisreal}
\left(1-p^{-m-1}\right)\sum_{e=0}^{\lfloor
\log_p(N/b)\rfloor}p^{me}+\bigo(1).
\end{equation}

If~$r_v=1$, then~$m=0$ and~\eqref{nothingisreal}
becomes,
\[
\left(1-\frac{1}{p}\right)
\log_p(N/b)+\bigo(1)=
\frac{p-1}{p\log p}\log N+\bigo(1),
\]

If~$r_v>1$, then~$m>0$ and~\eqref{nothingisreal}
becomes,
\begin{eqnarray*}
C\sum_{e=0}^{\lfloor
\log_p(N/b)\rfloor}p^{me}+\bigo(1)&=&
C\cdot
\frac{p^{m(1+\lfloor\log_p(N/b)\rfloor)}-1}{p^m-1}+\bigo(1)\\
&=&\delta(N)N^m+\bigo(1)
\end{eqnarray*}
for some constant~$C$ and a function~$\delta$ satisfying the
requirements of Theorem~\ref{rank_one_theorem}.

\section{Proof of Theorem~\ref{wheniwakeupearlyinthemorning}}

Since $X$ is connected, $\dirichlet_\alpha$ has a rational
expression which is a finite sum of terms of the
form~\eqref{manihadadreadfulflight}.
By Lemma~\ref{orshouldisay},
each such term may be written as
\[
F(z)
=
C\times G(z)\sum_{I\subset J}\frac{(-1)^{\vert I\vert}}{c_I}\,b_I^{-z},
\]
where
\[
G(z)=\prod_{p\in P}(1-p^{n_p-z})^{-1}
\]
for some $\mathbf{n}=(n_p)\in\mathbb{N}_0^P$. Let
\[
m=\max\{n_p:p\in P\}
\mbox{ and }
L=\left\vert\{p\in P\mid n_p=m\}\right\vert.
\]
Note that the Dirichlet series for~$G$ has abscissa of convergence~$m$ and~$G$
has a pole of order~$L$ at~$z=m$.

By showing that the coefficient
of~$(z-m)^{-L}$ in the Laurent series for $F$ is positive, it will follow
that for at least one such $F$, $m=\sigma$, $L=K$ and that there can be no $F$ for
which $m>\sigma$ nor any $F$ for which $m=\sigma$ and $L>K$.

The coefficient of~$(z-m)^{-L}$
in the Laurent series for $G$ about $m$ is~$\kappa=\prod_{p\in P}\kappa_p$
where
\[
\kappa_p
=
\begin{cases}
(\log p)^{-1} & \mbox{ if } n_p=m\\
(1-p^{n_p-m})^{-1} & \mbox{ if } n_p<m
\end{cases}
\]
Hence, the coefficient of~$(z-m)^{-L}$ in the Laurent series for~$F$ is
\begin{eqnarray*}
C \kappa\sum_{I\subset J}\frac{(-1)^{\vert I\vert}}{c_I b_I^m}
& = &
\frac{C\kappa}{B^m} \sum_{I\subset J}\frac{(-1)^{\vert I\vert}}{(\lcm\{i:i\in I\})^{m+1}}\\
& = &
\frac{C\kappa}{B^m} \sum_{I\subset J}\frac{(-1)^{\vert I\vert}}{(\lcm\{i^{m+1}:i\in I\})}\\
& = &
\frac{C\kappa}{B^m} \sum_{I\subset J'}\frac{(-1)^{\vert I\vert}}{(\lcm\{i:i\in I\})},
\end{eqnarray*}
where $J'=\{j^{m+1}:j\in J\}$. Since the constants $\kappa,B,C$ and the result of the sum are all
positive, so too is the coefficient of~$(z-m)^{-L}$ in the Laurent series for~$F$.

Now consider the contribution to~$\pi_\alpha$ arising from~$F$.
For any~$\mathbf{m}\in\mathbb{Z}^P$
and~$x\geqslant 0$, set
\[
S_x^P(\mathbf{m})=\mathop{\sum_{\mathbf{e}\in\mathbb{N}_0^P,}}_{\varphi_P(\mathbf{e})\leqslant x}
\varphi_P(\mathbf{m}\mathbf{e}),
\]
where integer vectors are multiplied term-by-term. Extracting the
coefficients from the Dirichlet series for~$F$, it follows that~$F$ contributes
\begin{equation}\label{ifyoutrytosit}
C\times
\sum_{I\subset J}\frac{(-1)^{\vert I\vert}}{c_I}
S_{N/b_I}^P(\mathbf{n})
\end{equation}
to $\pi_\alpha(N)$. This expression will be the main tool
for obtaining our asymptotics.

The following result gives a Chebychev estimate for
the individual terms $S_{N/b_I}^P(\mathbf{n})$.

\begin{lemma}\label{bethankfulidonttakeitall}
Let $\mathbf{m}=(m_p)\in\mathbb{Z}^P$ and $x\geqslant 0$.
Then there exist constants~$B\geqslant A>0$ and~$x_0\geqslant 0$
such that for all~$x>x_0$,
\[
A x^m(\log x)^j
\leqslant
S_x^P(\mathbf{m})
\leqslant
B x^m(\log x)^j,
\]
where $m=\max\{0,m_p:p\in P\}$ and
\[
j=
\begin{cases}
\left\vert\{p\in P\mid m_p=m\}\right\vert & \mbox{if }m=0,\\
\left\vert\{p\in P\mid m_p=m\}\right\vert-1 & \mbox{if }m>0.
\end{cases}
\]
\end{lemma}

\begin{proof}
First suppose $m=0$. For any $\mathbf{e}\in\mathbb{N}_0^P$,
\[
e_p\leqslant |P|^{-1}\log_p x\mbox{ for all }p\in P
\Rightarrow
\varphi_P(\mathbf{e})\leqslant x.
\]
Therefore,
\begin{eqnarray*}
S_x^P(\mathbf{m})
& \geqslant &
\prod_{p\in P}\sum_{e_p=0}^{\lfloor|P|^{-1}\log_p x\rfloor}p^{m_pe_p}\\
& \geqslant &
A(\log x)^j,
\end{eqnarray*}
for some positive constant $A$. On the other hand,
for any $\mathbf{e}\in\mathbb{N}_0^P$,
\[
\varphi_P(\mathbf{e})\leqslant x
\Rightarrow
e_p\leqslant \log_p x\mbox{ for all }p\in P.
\]
So,
\begin{eqnarray*}
S_x^P(\mathbf{m})
& \leqslant &
\prod_{p\in P}\sum_{e_p=0}^{\lfloor\log_p x\rfloor}p^{m_pe_p}\\
& \leqslant &
B(\log x)^j,
\end{eqnarray*}
for some positive constant $B\geqslant A$.

Now suppose $m>0$. The proof is by induction on $|P|$. If $|P|=1$ then the result is obvious, so assume $|P|>1$.
Choose $q\in P$ such that $m_q=m$, set~$P'=P\setminus\{q\}$
and~$\mathbf{m}'=(m_p)_{p\in P'}$. Write
\begin{equation}\label{shouldfivepercentappeartosmall}
S_x^P(\mathbf m) =
\mathop{\sum_{\mathbf{e}\in\mathbb{N}_0^{P'},}}
_{\varphi_{P'}(\mathbf{e})\leqslant x}
\varphi_{P'}(\mathbf{m'}\mathbf{e})
\sum_{e_q=0}^{\lfloor\log_q (x/\varphi_{P'}(\mathbf{e}))\rfloor}
q^{m e_q}.
\end{equation}
The inner sum in~\eqref{shouldfivepercentappeartosmall} is a geometric series
with sum
\[
\frac{q^{m(1+\lfloor\log_q
x/\varphi_{P'}(\mathbf{e})\rfloor)}-1}{q^{m}-1}.
\]
This expression is bounded above by
\[
C_1 x^m\varphi_{P'}(-m\mathbf{e})-C_2
\]
for some constants $C_1,C_2>0$. It
follows from~\eqref{shouldfivepercentappeartosmall} and the $m=0$ case that
\begin{eqnarray*}
S_x^P(\mathbf{m})
& \leqslant &
C_1 x^m S_x^{P'}(\mathbf{m}'-m\mathbf{1})-C_2 S_x^{P'}(\mathbf{m}')\\
& \leqslant &
C_3 x^m (\log x)^j-C_2 S_x^{P'}(\mathbf{m}'),
\end{eqnarray*}
for some constant~$C_3>0$. Upon noting that either
\begin{enumerate}
\item $\max\{m_p':p\in P'\}<m$ or
\item $\max\{m_p':p\in P'\}=m$
 and $\left\vert\{p\in P'\mid m_p'=m\}\right\vert-1<j$,
\end{enumerate}
this completes the inductive step. A similar argument also gives the lower bound.
\end{proof}

\subsection{Proof of Theorem~\ref{wheniwakeupearlyinthemorning}(1)}\label{theyvebeengoinginandoutofstyle}

In this case all expressions of the form~\eqref{ifyoutrytosit}
contributing to $\pi_\alpha(N)$ have $\mathbf{n}=\mathbf{0}$. Example~\ref{exampletwo} illustrates
some of the issues that arise in this setting when more than one prime is involved.

\begin{example}\label{threecapfive}
(Example~\ref{exampletwo} revisited)
Let~$\alpha$ be the map dual to~$x\mapsto 2x$ on~$\mathbb{Z}_{(3)}\cap\mathbb{Z}_{(5)}$, so
\[
\fix_{\alpha}(n)=\vert2^n-1\vert_3^{-1}\vert2^n-1\vert_5^{-1}.
\]
The term in~$\dirichlet_{\alpha}$ that determines the
asymptotic growth in~$\pi_{\alpha}$ is
\[
15\left(
\frac{1}{4^{z+1}}-\frac{1}{12^{z+1}}-\frac{1}{20^{z+1}}
+\frac{1}{60^{z+1}}
\right)
\frac{1}{(1-3^{-z})(1-5^{-z})}.
\]
Notice that each term of the form
\[
\frac{b^{-z}}{(1-3^{-z})(1-5^{-z})}
\]
with~$b>1$ contributes
\begin{eqnarray*}
S_N
& = &
\sum_{e_1=1}^{\lfloor\log_3(N/b)\rfloor}
\underbrace{\sum_{e_2=0}^{\lfloor\log_5(N/b3^{e_1})\rfloor}
1}_{=\frac{\log N}{\log5}-\frac{\log3}{\log5}e_1+\bigo(1)}\\
& = &
\frac{\log N}{\log5}\sum_{e_1=1}^{\lfloor\log_3(N/b)\rfloor} 1
-\frac{\log3}{\log5}
\sum_{e_1=1}^{\lfloor\log_3(N/b)\rfloor} e_1
+\bigo(\log N)\\
& = &
\frac{(\log N)^2}{\log3\log5}-\frac{\log3}{2\log5}
\lfloor\log_3(N/b)\rfloor^2
+ \bigo(\log N)\\
& = &
\frac{(\log N)^2}{\log3\log5}\left(1-\frac{1}{2}\right)
+ \bigo(\log N)
\end{eqnarray*}
to~$\pi_{\alpha}(N)$.
Summing over all the terms therefore gives
\begin{eqnarray*}
\pi_{\alpha}(N)&=&
\frac{15}{4}\left(
1-\frac{1}{3}-\frac{1}{5}+\frac{1}{15}
\right)
\frac{(\log N)^2}{2\log3\log5}
+\bigo(\log N)\\
&=&\frac{(\log N)^2}{\log3\log5}+\bigo(\log N).
\end{eqnarray*}
\end{example}

Returning to the general case, we will use
an induction argument to obtain an exact
asymptotic for~$S_x^P(\mathbf{0})$. The following lemma
provides the essential inductive step, but is more general than we need at this stage
(the full statement being needed for the proof of
Theorem~\ref{wheniwakeupearlyinthemorning}(3)). The result is
a little technical and requires some preparation.
Suppose~$\mathbf{m}=(m_p)\in\mathbb{Z}^P$
satisfies~$m_p\leqslant 0$ for all~$p\in P$. Set
\begin{equation}\label{guaranteedtoraiseasmile}
\overline{P}=\{p\in P:m_p<0\}.
\end{equation}
Let $\overline{P}\subseteq W \subseteq P$, $\mathbf{d}\in\mathbb{N}_0^W$
and define
$\widetilde{\mathbf{d}}=(\widetilde{d_p})\in\mathbb{N}_0^W$ by
\begin{equation}\label{somayiintroducetoyou}
\widetilde{d_p}=
\begin{cases}
0 & \mbox{ if } p\in\overline{P},\\
d_p & \mbox{ if } p\not\in\overline{P}.
\end{cases}
\end{equation}
Set
\begin{equation}\label{sgtpeppertaughtthebandtoplay}
T_x^W(\mathbf{m},\mathbf{d},k)
=
(\log x)^{-k}
\mathop{\sum_{\mathbf{e}\in\mathbb{N}_0^W,}}_{\varphi_W(\mathbf{e})\leqslant x}
\varphi_W(\mathbf{m}'\mathbf{e})\mathbf{e}^{\mathbf{d}},
\end{equation}
where $\mathbf{m}'=(m_p)_{p\in W}$.

Notice that
\begin{eqnarray}
\nonumber
T_x^W(\mathbf{m},\mathbf{d},k)
& \leqslant &
(\log x)^{-k}\prod_{p\in W}\sum_{e_p=0}^{\lfloor\log_p x\rfloor}p^{m_p e_p}e_p^{d_p}\\
\nonumber
& \leqslant &
C(\log x)^{-k}\prod_{p\in W\setminus\overline{P}}(\log x)^{d_p+1}\\
\label{illtaxyourfeet}
& = &
C(\log x)^{\widetilde{\mathbf{d}}\cdot\mathbf{1}+|W\setminus\overline{P}|-k}
\end{eqnarray}

\begin{lemma}\label{nowmyadviceforthosewhodie}
Suppose that $W\neq\overline{P}$ and $W'=W\setminus\{s\}$ for some $s\in W\setminus\overline{P}$.
If $\widetilde{\mathbf{m}}=\mathbf{0}$, $\mathbf{d}=\widetilde{\mathbf{d}}$,  $k=\mathbf{d}\cdot\mathbf{1}+|W\setminus\overline{P}|$ then there
exist a finite indexing set $H$ and set of constants $\{c_h:h\in H\}$,
both independent of~$x$, such that
\[
T_x^W(\mathbf{m},\mathbf{d},k)
=
\left(
\sum_{h\in H}
c_h T_x^{W'}(\mathbf{m},\mathbf{d}(h),k_h)
\right)
+\bigo\left(\frac{1}{\log x}\right),
\]
where for each $h\in H$, $\mathbf{d}(h)\in\mathbb{N}_0^{W'}$ and $k_h\in\mathbb{N}_0$ satisfy
\begin{enumerate}
\item $\mathbf{d}(h)=\widetilde{\mathbf{d}(h)}$,
\item $k_h=\mathbf{d}(h)\cdot\mathbf{1}+|W'\setminus\overline{P}|$.
\end{enumerate}
\end{lemma}

\begin{proof}
If $W'=\varnothing$ the result is obvious, so assume $W'\neq\varnothing$.
By assumption $m_s=0$, so
\begin{equation}\label{dontaskmewhatiwantitfor}
T_x^W(\mathbf{m},\mathbf{d},k)
=
(\log x)^{-k}
\mathop{\sum_{\mathbf{e}\in\mathbb{N}_0^{W'},}}_{\varphi_{W'}(\mathbf{e})\leqslant x}
\varphi_{W'}(\mathbf{m}'\mathbf{e})\mathbf{e}^{\mathbf{d}'}
\sum_{e_s=0}^{\lfloor\log_s(x/\varphi_{W'}(\mathbf{e}))\rfloor}e_s^{d_s},
\end{equation}
where $\mathbf{m}'=(m_p)_{p\in W'}$ and $\mathbf{d}'=(d_p)_{p\in W'}$.
Expanding the inner sum,
\begin{eqnarray}
\nonumber
\sum_{e_s=0}^{\lfloor\log_s (x/\varphi_{W'}(\mathbf e))\rfloor}
e_s^{d_s}
& = &
C\lfloor\log_s x/\varphi_{W'}(\mathbf e)\rfloor^{d_s+1}+
\bigo\left((\log x)^{d_s}\right)\\
\nonumber
& = &
C\left(\log_s x/\varphi_{W'}(\mathbf e)\right)^{d_s+1}+
\bigo\left((\log x)^{d_s}\right)\\
\nonumber
& = &
C\left(\log x-\mathbf e\cdot(\log p)_{p\in W'}\right)^{d_s+1}+
\bigo\left((\log x)^{d_s}\right)\\
\label{ifyoutakeawalk}
& = &
\sum_{h\in H}c_h(\log x)^{k_h'}\mathbf e^{\mathbf{d}'(h)}
+\bigo\left((\log x)^{d_s}\right),
\end{eqnarray}
where the finite indexing set $H$ and set of constants $\{c_h:h\in H\}$ are
both independent of $x$, and $k_h\in\mathbb{N}_0$, $\mathbf{d}'(h)\in\mathbb{N}_0^{W'}$ satisfy
\begin{equation}\label{ifyougettoocold}
k_h'+\mathbf{d}'(h)\cdot\mathbf{1}=d_s+1
\end{equation}
Let $k_h=k-k_h'$ and $\mathbf{d}(h)=\mathbf{d}'+\mathbf{d}'(h)$. Substituting~\eqref{ifyoutakeawalk}
into~\eqref{dontaskmewhatiwantitfor} gives
\begin{equation}\label{ifyoudontwanttopaysomemore}
T_x^W(\mathbf{m},\mathbf{d},k)
=
\left(
\sum_{h\in H'}
c_h T_x^{W'}(\mathbf{m},\mathbf{d}(h),k_h)
\right)
+\bigo\left(\frac{1}{\log x}\right);
\end{equation}
the error term being obtained using~\eqref{illtaxyourfeet}, since
\begin{eqnarray*}
(\log x)^{d_s-k}
\mathop{\sum_{\mathbf{e}\in\mathbb{N}_0^{W'},}}_{\varphi_{W'}(\mathbf{e})\leqslant x}
\varphi_{W'}(\mathbf{m}'\mathbf{e})\mathbf{e}^{\mathbf{d}'}
& = &
\bigo
\left(
(\log x)^{{\widetilde{\mathbf{d}'}\cdot\mathbf{1}+|W'\setminus\overline{P}|+d_s-k}}
\right)\\
& = &
\bigo
\left(
(\log x)^{{\widetilde{\mathbf{d}}\cdot\mathbf{1}+|W\setminus\overline{P}|-k-1}}
\right)\\
& = &
\bigo
\left(
(\log x)^{k-k-1}
\right).
\end{eqnarray*}

It remains to see that conditions (1) and (2) are satisfied. To see that (2)
holds, note~\eqref{ifyougettoocold} implies
\begin{eqnarray*}
k-k_h+(\mathbf{d}(h)-\mathbf{d}')\cdot\mathbf{1}
& = &
d_s+1\\
\Rightarrow
k_h
& = &
\mathbf{d}(h)\cdot\mathbf{1}+k-\mathbf{d}\cdot\mathbf{1}-1\\
& = &
\mathbf{d}(h)\cdot\mathbf{1}+|W'\setminus\overline{P}|,
\end{eqnarray*}
since by hypothesis $k-\mathbf{d}\cdot\mathbf{1}=|W\setminus\overline{P}|$.

To satisfy condition~(1),
we show that~$H$ may simply be replaced by
\[
\{h\in H:\mathbf{d}(h)=\widetilde{\mathbf{d}(h)}\},
\]
in~\eqref{ifyoudontwanttopaysomemore}.
To see this, note that
if~$\mathbf{d}(h)\neq\widetilde{\mathbf{d}(h)}$
then~$\widetilde{\mathbf{d}(h)}\cdot 1\leqslant\mathbf{d}(h)\cdot\mathbf{1}-1$ and so
\begin{eqnarray*}
\widetilde{\mathbf{d}(h)}\cdot
\mathbf{1}+|W'\setminus\overline{P}|-k_h
& \leqslant &
\mathbf{d}(h)\cdot\mathbf{1}+|W'\setminus
\overline{P}|-k_h-1\\
& \leqslant &
-1,
\end{eqnarray*}
since (2) holds. Therefore,
\[
T_x^{W'}(\mathbf{m},\mathbf{d}(h),k_h) =
\bigo\left(\frac{1}{\log x}\right)
\]
by \eqref{illtaxyourfeet}.
\end{proof}

Let $L=|P|$. Since
\[
(\log x)^{-L}S_x^P(\mathbf{0})=T_x^{W'}(\mathbf{0},\mathbf{0},L),
\]
applying Lemma~\ref{nowmyadviceforthosewhodie} and induction, we obtain
\[
(\log x)^{-L}S_x^P(\mathbf{0}) = C_1+\bigo\left(\frac{1}{\log x}\right)
\]
for some constant $C_1$ which is independent of $x$.
Notice that $C_1>0$, otherwise there would be a
contradiction to the lower bound
provided by Lemma~\ref{bethankfulidonttakeitall}.
Furthermore,
\begin{eqnarray*}
S_{N/b_I}^P(\mathbf{0})
& = &
C_1(\log (N/b_I))^L + \bigo((\log (N/b_I))^{L-1})\\
& = &
C_1(\log N)^L + \bigo((\log N)^{L-1})
\end{eqnarray*}
It follows that the contribution to~$\pi_\alpha(N)$ from~\eqref{ifyoutrytosit} is
\begin{eqnarray*}
C_2
\left(\sum_{I\subset J}\frac{(-1)^{\vert I\vert}}{c_I}\right)
(\log N)^L + \bigo((\log N)^{L-1}),
\end{eqnarray*}
where $C_2>0$ and also~$\sum_{I\subset J}(-1)^{\vert I\vert}/c_I>0$.

Since there is at least
one term of the form~\eqref{ifyoutrytosit} contributing to~$\pi_\alpha(N)$ for which~$L=K$,
this completes the proof.

\subsection{Proof of Theorem~\ref{wheniwakeupearlyinthemorning}(2)}

Recall that in this case, for each term of the form~\eqref{ifyoutrytosit} contributing
to~$\pi_\alpha(N)$, $\mathbf{n}=(n_p)$ satisfies
\[
\max\{n_p:p\in P\}\leqslant\sigma.
\]
Moreover, since~$\sigma$ is a simple pole,
any term with $\max\{n_p:p\in P\}=\sigma$ has at most one of
the corresponding~$n_p$ equal to $\sigma$. The upper bound now follows from
Lemma~\ref{bethankfulidonttakeitall}.

For the lower bound we use more elementary methods.
The orbits for these systems mainly accumulate
at specific lengths, and summing over orbits
of those lengths suffices. Recall that~$\orbit_{\alpha}(n)$
is related to the number of periodic
points by~\eqref{equation:moebius_inversion_formula}
where~$\fix_{\alpha}(n)$ is given by the
formula~\eqref{Youdontknowhowluckyyouluckyyouareboy}
with~$S=T$. By~\eqref{manihadadreadfulflight},
there is a prime~$q\in Q$ such that
\begin{equation}\label{declarethepenniesonyoureyes}
\sum_{v\in U}r_v=\sigma+1,
\end{equation}
where~$U=\{v\in T\mid v\divides q\}$.

Using the notation from Section~\ref{sectionproofofrationality},
let~$\morder=\lcm\{\morder_v\mid v\in U\}$. Notice
that~$\gcd(q,\morder)=1$ since~$\morder_v$
divides~$q^{r_v}-1$ for~$v\in U$. Using basic properties of the
the M{\"o}bius function,
\begin{eqnarray}
\orbit_{\alpha}(\morder q^e)&=&\frac{1}{\morder q^e}
\sum_{e'=0}^{e}\sum_{d\divides\morder}
\mu(q^{e-e'}\morder/d)f_T(dq^{e'})\nonumber\\
&=&\frac{1}{\morder q^e}\sum_{d\divides\morder}
\mu(\morder/d)\left(
f_T(dq^e)-f_T(dq^{e-1})
\right)\nonumber\\
&=&\frac{1}{\morder q^e}\left(
f_T(\morder q^e)-f_T(\morder q^{e-1})
\right)
\nonumber\\
& &
\medspace
\medspace\medspace\medspace
\medspace\medspace\medspace
\medspace\medspace\medspace
+
\frac{1}{\morder q^e}
\sum_{d\divides\morder,d<\morder}
\mu(\morder/d)\left(
f_T(dq^e)-f_T(dq^{e-1})
\right).\nonumber
\end{eqnarray}

By \eqref{declarethepenniesonyoureyes} and
Lemma~\ref{fundamental_evaluation_lemma},
there are constants~$D,E>0$ such that
\[
f_T(\morder q^e)=Dq^{(\sigma+1)e}
\qquad\mbox{ and }\qquad
f_T(dq^e)\le Eq^{\sigma e}
\]
for all~$d<\morder$ and large enough~$e$. Now,
\begin{eqnarray*}
\orbit_{\alpha}(\morder q^e)
&\ge&
\frac{1}{\morder q^e}\left(\vphantom{\sum}
\left(
Dq^{(\sigma+1)e}-Dq^{(\sigma+1)(e-1)}
\right)
-2E\morder q^{\sigma e}
\right)\\
& = &
\frac{D}{\morder}\left(1-{q^{-(\sigma+1)}}\right)q^{\sigma e}
-2E q^{(\sigma-1) e},
\end{eqnarray*}
so there are constants $C_1\ge 0$ and $C_2,C_3,C_4>0$ such that
\begin{eqnarray*}
\pi_{\alpha}(N)
& \ge &
\left(\sum_{e=0}^{\lfloor\log_q N/\morder\rfloor}\orbit_{\alpha}
(\morder q^e)\right)-C_1\\
& \ge &
\left(\sum_{e=0}^{\lfloor\log_q N/\morder\rfloor}
C_2q^{\sigma e}-C_3q^{(\sigma-1)e}\right)-C_1\\
& \ge &
C_4 N^{\sigma},
\end{eqnarray*}
for all large~$N$.

\subsection{Proof of Theorem~\ref{wheniwakeupearlyinthemorning}(3)}

We begin with an example that illuminates some of the issues that arise in this case.

\begin{example}\label{theend?}
Consider~$M=\mathbb{Z}_{(3)}^2\times \mathbb{Z}_{(5)}^2$,~$X
=\widehat{M}$ and the the endomorphism~$\alpha:X\rightarrow X$
given by~$x\mapsto 2x$. The dynamical system~$(X,\alpha)$
has periodic point data identical to that of a
product of Example 4.2 and another duplicate
system. Consequently,
\[
\fix_{\alpha}(n)=
\begin{cases}
1&\mbox{if $n$ is odd};\\
9\cdot3^{2e_1}&\mbox{if }n=2\cdot k\cdot 3^{e_1},
2\notdivides k,3\notdivides k;\\
225\cdot 3^{2e_1}\cdot 5^{2e_2}&\mbox{if }n=4\cdot k\cdot3^{e_1}
\cdot 5^{e_2},3\notdivides k,
5\notdivides k.
\end{cases}
\]
Following a similar method to that used in Example~\ref{exampletwo},
the term of~$\dirichlet_\alpha$ dominating the growth of~$\pi_\alpha$ is
\begin{equation}\label{letmetellyouhowitwillbe}
\left(
\frac{1}{4^{z+1}}-\frac{1}{12^{z+1}}-\frac{1}{20^{z+1}}
+\frac{1}{60^{z+1}}
\right)
\frac{225}{(1-3^{1-z})(1-5^{1-z})}.
\end{equation}
Hence, consider
\[
S_{N/b}=\mathop{\sum_{\mathbf{e}\in
\mathbb{N}_0^2,}}_{3^{e_1}5^{e_2} \leqslant N/b}3^{e_1}5^{e_2}.
\]
Then, writing~$\{\cdot\}$ for the fractional part,
\begin{eqnarray*}
\frac{S_{N/b}}{N\log N}
& = &
\frac{1}{N\log N}
\mathop{\sum_{\mathbf{e}\in\mathbb{N}_0^2,}}_{3^{e_1}5^{e_2}\leqslant N/b}3^{e_1}5^{e_2}\\
& = &
\frac{1}{N\log N}
\sum_{e_1=0}^{\lfloor\log_3 (N/b)\rfloor} 3^{e_1}
\sum_{e_2=0}^{\lfloor\log_5 (N/3^{e_1}b)\rfloor} 5^{e_1}\\
& = &
\frac{1}{N\log N}
\sum_{e_1=0}^{\lfloor\log_3 N/b\rfloor} 3^{e_1}
\left(\frac{5N}{4b}\,3^{-e_1} 5^{-\{\log_5(N/3^{e_1}b)\}}-\frac{1}{4}\right)\\
& = &
\frac{5}{4b\log N}
\sum_{e_1=0}^{\lfloor\log_3 N/b\rfloor}
5^{-\{\log_5 N/b - e_1 \log_5 3\}}
+\bigo(N^{-1}).
\end{eqnarray*}
Since the exponent in the inner sum is
uniformly distributed in the unit interval,
by Weyl's theorem~\cite{MR1511862},
\begin{eqnarray*}
\frac{1}{\log N}\sum_{e_1=0}^{\lfloor\log_3 N/b\rfloor}
5^{-\{\log_5 N/b - e_1 \log_5 3\}}
& \rightarrow &
\frac{1}{\log 3}
\int_{0}^{1}5^{-y}\,dy\\
& = &
\frac{4}{5\log 3\log 5}.
\end{eqnarray*}
Therefore,
\[
\frac{S_{N/b}}{N\log N}\rightarrow\frac{1}{b\log 3\log 5}.
\]
From the above and \eqref{letmetellyouhowitwillbe}, it follows that
\begin{eqnarray*}
\frac{\pi_\alpha(N)}{N\log N}
& \sim &
\frac{225}{N\log N}
\left(
\frac{S_{N/4}}{4}-\frac{S_{N/12}}{12}-\frac{S_{N/20}}{20}+\frac{S_{N/60}}{60}
\right)\\
& \sim &
\frac{225}{\log 3\log 5}\left(\frac{1}{4^2}-\frac{1}{12^2}-\frac{1}{20^2}+\frac{1}{60^2}\right)\\
& = &
\frac{12}{\log 3\log 5}.
\end{eqnarray*}
\end{example}

In the general case, the proof uses similar ideas,
with the main steps being equidistribution and
an inclusion-exclusion argument. As in the previous section,
for each term of the form~\eqref{ifyoutrytosit} contributing
to~$\pi_\alpha(N)$, $\mathbf{n}=(n_p)$ satisfies
\[
\max\{n_p:p\in P\}\leqslant\sigma,
\]
and if $\max\{n_p:p\in P\}=\sigma$ then
\[
|\{p\in P:n_p=\sigma\}|\leqslant K.
\]
Furthermore, there is at least one term with
\begin{equation}\label{youreworkingfornoonebutme}
\max\{n_p:p\in P\}=\sigma
\mbox{ and }
|\{p\in P:n_p=\sigma\}|=K
\end{equation}
In light of the upper bound provided by Lemma~\ref{bethankfulidonttakeitall}, terms
of this form give the dominant contribution to $\pi_\alpha(N)$. Thus,
to obtain the required asymptotic it suffices to prove that
\begin{equation}\label{itwastwentyyearsagotoday}
\frac{1}{N^\sigma(\log N)^{K-1}}
\sum_{I\subset J}\frac{(-1)^{\vert I\vert}}{c_I}
S_{N/b_I}^P(\mathbf{n})
\end{equation}
convergences to a positive constant when~$\mathbf{n}$
satisfies~\eqref{youreworkingfornoonebutme}.

Choose a prime~$q\in P$ with~$n_q=\sigma$ and set~$P'=P\setminus\{q\}$.
Write
\[
a_N=\frac{1}{N^{\sigma}(\log N)^{K-1}}S_{N/b}^{P}(\mathbf n).
\]
Then
\[
a_N
=
\frac{1}{N^{\sigma}(\log N)^{K-1}}
\sum_{\varphi_{P'}(\mathbf e)\le N/b}
\varphi_{P'}(\mathbf n'\mathbf e)
\sum_{e_q=0}^{\lfloor
\log_q(N/b\varphi_{P'}(\mathbf e))\rfloor}\negmedspace\negmedspace
q^{\sigma e_q}
\]
where~$\mathbf n'=(n_p)_{p\in P'}$.
Treating the inner sum as a geometric progression as
usual, we see that~$a_N$ is equal to
\[
\frac{C_1}{b^{\sigma}(\log N)^{K-1}}
\left(
\sum_{\varphi_{P'}(\mathbf e)\le N/b}
\varphi_{P'}(\mathbf n'\mathbf e)
\varphi_{P'}(\sigma\mathbf e)^{-1}
q^{-\sigma\{\log_q(N/b\varphi_{P'}(\mathbf e))\}}
\right)
+\Delta_N
\]
where~$C_1=1/(1-q^{-\sigma})$ and
\[
\Delta_N
=
\bigo\left(
\frac{1}{N^{\sigma}(\log N)^{K-1}}S_{N/b}^{P'}(\mathbf n')
\right)
=\bigo
\left(
\frac{1}{\log N}
\right)
\]
by Lemma~\ref{bethankfulidonttakeitall}.

Now let~$r\in P'$ be a prime with~$m_r=\sigma$
(such a prime exists by hypothesis) and set $W=P\setminus\{q,r\}$.
Then~$a_N$ is, up to~$\bigo(1/\log N)$,
\begin{eqnarray*}
\frac{C_1}{b^{\sigma}(\log N)^{K-1}}
\sum_{\varphi_W(\mathbf e)\le N/b}
\varphi_W(\mathbf m\mathbf e)
\sum_{e_r=0}^{\lfloor
\log_r(N/b\varphi_W(\mathbf e))
\rfloor}
q^{-\sigma\{
\log_q(N/b\varphi_W(\mathbf e)r^{e_r})
\}}
\end{eqnarray*}
where $\mathbf m=(m_p)\in\mathbb{Z}^W$ is given by
\[
m_p=
\begin{cases}
n_p-\sigma&\mbox{if }n_p<\sigma,\\
0&\mbox{if }n_p=\sigma.
\end{cases}
\]
For $\varphi_W(\mathbf e)\leqslant x$, set
\[
\mathcal{I}_x(\mathbf e)=
\frac{1}{\log_r(x/\varphi_W(\mathbf e))}
\sum_{e_r=0}^{\lfloor\log_r(x/\varphi_W(\mathbf e))\rfloor}
q^{-{\sigma\{\log_q(x/\varphi_W(\mathbf e)r^{e_r})\}}}
\]
and
\[
\mathcal{J}_x(\mathbf{e})=
\varphi_W(\mathbf m\mathbf e)\log(x/\varphi_W(\mathbf e)).
\]
Then, again up to~$\bigo(1/\log N)$,~$a_N$ is
\[
\frac{C_2}{b^{\sigma}(\log N)^{K-1}}
\sum_{\varphi_W(\mathbf e)\le N/b}
\mathcal{J}_{N/b}(\mathbf{e})
\mathcal{I}_{N/b}(\mathbf{e}),
\]
where $C_2=C_1/\log r$. Since~$r\neq q$ the circle rotation by~$\log_qr$ is
uniquely ergodic, so there is uniform convergence for
continuous functions in the ergodic theorem
(see Oxtoby~\cite[\S5]{MR0047262}
for example). Hence, given~$\epsilon>0$
there is an~$M$ such that whenever~$N/b\varphi_W(\mathbf e)\ge M$,
\[
\mathcal{I}_{N/b}
(\mathbf e)=\int_{0}^{1}
q^{-\sigma y}\,dy+\delta_{N/b}(\mathbf e),
\]
where~$|\delta_{N/b}(\mathbf e)|<\epsilon$.
For~$N/b\varphi_W(\mathbf e)<M$ (that is in the range $N/bM<\varphi_W(\mathbf e)\leqslant N/b$) there exists a uniform constant~$D$
for which
\[
\mathcal{I}_{N/b}(\mathbf e)=\int_{0}^{1}
q^{-\sigma y}\,dy+\beta_{N/b}(\mathbf e)
\]
where~$|\beta_{N/b}(\mathbf e)|<D$.

Let
\begin{eqnarray*}
b_N
&=&
\sum_{\varphi_W(\mathbf e)\le N/b}
\mathcal{J}_{N/b}(\mathbf{e}),\\
c_N
&=&
\sum_{\varphi_W(\mathbf e)\le N/bM}
\mathcal{J}_{N/b}(\mathbf{e})
\delta_{N/b}(\mathbf e),\mbox{ and }\\
d_N
&=&
\sum_{N/bM<\varphi_W(\mathbf e)\le N/b}
\mathcal{J}_{N/b}(\mathbf{e})
\beta_{N/b}(\mathbf e).
\end{eqnarray*}
Then
\[
a_N=\frac{1}{b^{\sigma}}(\log N)^{1-K}(C_3b_N+C_2c_N+C_2d_N)+\bigo\left(\frac{1}{\log N}\right),
\]
where $C_3>0$ is independent of~$b$.

\begin{lemma}\label{geeitsgoodtobebackhome}
We have the following estimates:
\begin{enumerate}
\item $(\log N)^{1-K}b_N$ converges to a constant~$C_4 \geqslant 0$
that is independent of~$b$;
\item $(\log N)^{1-K}|c_N|\le C_5\epsilon$; and
\item $(\log N)^{1-K}d_N$ converges to zero as~$N\to\infty$.
\end{enumerate}
\end{lemma}

Assuming Lemma~\ref{geeitsgoodtobebackhome} for now and setting $C_6=C_3C_4$,
we deduce that
\begin{equation}\label{theukrainegirlsreallyknowckmeout}
a_N\longrightarrow\frac{C_6}{b^{\sigma}}\mbox{ as }
N\to\infty
\end{equation}
This can be used to show that~\eqref{itwastwentyyearsagotoday} converges to
a positive constant, therefore completing the proof.
To see this, first note that if $C_6=0$ then~\eqref{theukrainegirlsreallyknowckmeout}
contradicts the lower bound provided by Lemma~\ref{bethankfulidonttakeitall}, therefore~$C_6>0$.
Furthermore,~\eqref{itwastwentyyearsagotoday} converges to
\begin{eqnarray*}
C_6\sum_{I\subset J}\frac{(-1)^{\vert I\vert}}{c_Ib_I^{\sigma}}
&=&
C_6\sum_{I\subset J}\frac{(-1)^{\vert I\vert}}{c_I(Bc_I)^{\sigma}}\\
&=&
\frac{C_6}{B^\sigma}
\sum_{I\subset J'}\frac{(-1)^{\vert I\vert}}
{\lcm\{i\mid i\in I\}}\nonumber\\
&>&0,\nonumber
\end{eqnarray*}
where~$J'=\{i^{\sigma+1}\mid i\in J\}$.

Finally we turn to the postponed proof of the lemma.

\begin{proof}[Proof of Lemma~\ref{geeitsgoodtobebackhome}]
To prove (1), first note we may write
\begin{eqnarray*}
(\log N)^{1-K}b_N
& = &
\frac{1}{(\log(N/b) + \log b)^{K-1}}\,b_N\\
& = &
(1+\Delta_N)^{-1}(\log (N/b))^{1-K}b_N
\end{eqnarray*}
where $\Delta_N=\bigo((\log N)^{-1})$. Therefore, it suffices to prove
\[
(\log (N/b))^{1-K}b_N\rightarrow C_4.
\]
Write $x=N/b$ and consider
\begin{eqnarray*}
(\log (N/b))^{1-K}b_N
& = &
(\log x)^{1-K}
\sum_{\varphi_W(\mathbf e)\le x}\mathcal{J}_x(\mathbf{e})\\
& = &
(\log x)^{1-K}
\sum_{\varphi_W(\mathbf e)\le x}
\varphi_W(\mathbf m\mathbf e)\log(x/\varphi_W(\mathbf e))\\
& = &
(\log x)^{1-K}
\sum_{\varphi_W(\mathbf e)\le x}
\varphi_W(\mathbf m\mathbf e)(\log x - \mathbf{e}\cdot(\log p)_{p\in W}).
\end{eqnarray*}
Therefore, $(\log (N/b))^{1-K}b_N$ may be written as a sum of $|W|+1$ expressions of the form
\[
(\log x)^{-k}
\sum_{\varphi_W(\mathbf e)\le x}
\varphi_W(\mathbf m\mathbf e)\mathbf{e}^{\mathbf{d}},
\]
where $\mathbf{d}\in\mathbb{N}_0^W$ satisfies
\begin{equation}\label{theactyouveknown}
k - \mathbf{d}\cdot \mathbf{1} = K-2
= |W\setminus\overline{P}|,
\end{equation}
with $\overline{P}$ defined by \eqref{guaranteedtoraiseasmile}.

So, we must consider an expression of the form $T_x^W(\mathbf{m},\mathbf{d},k)$,
defined by~\eqref{sgtpeppertaughtthebandtoplay}. Let~$\widetilde{\mathbf{d}}$ be defined by~\eqref{somayiintroducetoyou}.
If~$\mathbf{d}\neq\widetilde{\mathbf{d}}$
then~$\widetilde{\mathbf{d}}\cdot\mathbf{1}<\mathbf{d}\cdot\mathbf{1}$.
Hence by~\eqref{illtaxyourfeet} and~\eqref{theactyouveknown},
\[
T_x^W(\mathbf{m},\mathbf{d},k) = \bigo\left(\frac{1}{\log x}\right).
\]
Thus, we need only consider those terms~$T_x^W(\mathbf{m},\mathbf{d},k)$
with~$\mathbf{d}=\widetilde{\mathbf{d}}$.
If $\overline{P}=\varnothing$,
then proceeding as in~Section~\ref{theyvebeengoinginandoutofstyle}, an induction using Lemma~\ref{nowmyadviceforthosewhodie} shows that
\[
T_x^W(\mathbf{m},\mathbf{d},k)=C_7+\bigo\left(\frac{1}{\log x}\right),
\]
where~$C_7$ is a constant independent of~$x$. If~$\overline{P}\neq\varnothing$ then again by induction using Lemma~\ref{nowmyadviceforthosewhodie}, up to an error of $\bigo(1/\log x)$, $T_x^W(\mathbf{m},\mathbf{d},k)$ may be
written as a linear combination of expressions of the form
\[
T_x^{\overline{P}}(\mathbf{m},\mathbf{0},0)=
\sum_{\varphi_{\overline{P}}(\mathbf{e})\leqslant x}
\varphi_{\overline{P}}(\mathbf{m}'\mathbf{e})
\]
where $\mathbf{m}'=(m_p)_{p\in\overline{P}}$. Moreover,
\[
\prod_{p\in\overline{P}}
\sum_{p=0}^{\lfloor|\overline{P}|^{-1}\log_p x\rfloor}p^{m_pe_p}
\leqslant
T_x^{\overline{P}}(\mathbf{m},\mathbf{0},0)
\leqslant
\prod_{p\in\overline{P}}
\sum_{p=0}^{\lfloor\log_p x\rfloor}p^{m_pe_p},
\]
so
\[
T_x^{\overline{P}}(\mathbf{m},\mathbf{0},0)\longrightarrow\prod_{p\in\overline{P}}\frac{1}{1-p^{m_p}}.
\]
as $x\rightarrow\infty$.

Thus, $(\log (N/b))^{1-K}b_N$ converges to a constant $C_4$ independent of~$b$ which is non-negative, as
each term of the sequence is non-negative.

To prove (2), simply note that $|c_N|\leqslant b_N\varepsilon$. To prove (3), first notice
\begin{eqnarray*}
|d_N|
& \leqslant &
D(\log N)^{1-K}
\sum_{N/bM<\varphi_W(\mathbf e)\le N/b}
\mathcal{J}_{N/b}(\mathbf{e})\\
& = &
D(\log N)^{1-K}
\left(
b_N -
\sum_{\varphi_W(\mathbf e)\le N/bM}
\mathcal{J}_{N/b}(\mathbf{e})
\right).
\end{eqnarray*}
But
\[
\mathcal{J}_{N/b}(\mathbf{e})
=
\mathcal{J}_{N/bM}(\mathbf{e})
+(\log M)\varphi_W(\mathbf{m}\mathbf{e}),
\]
so
\begin{eqnarray*}
|d_N|
& \leqslant &
D(\log N)^{1-K}(b_N-b_{N/M})
-
D(\log M)(\log N)^{1-K}
S_{N/bM}^W(\mathbf{m})\\
& = &
\frac{D}{(\log N)^{1-K}}(b_N-b_{N/M})
+
\bigo\left(\frac{1}{\log N}\right),
\end{eqnarray*}
by Lemma~\ref{bethankfulidonttakeitall}. Since~$(\log N)^{1-K}b_N$
and~$(\log N)^{1-K}b_{N/M}$ both converge to~$C_4$, (3) follows.
\end{proof}


\end{document}